\documentclass[12pt,reqno,english,a4]{amsart}
\usepackage{amsmath,amsthm,amsfonts,amssymb,amscd}
\usepackage[latin1]{inputenc}
\usepackage{psfrag}
\usepackage{epsfig}
\usepackage{a4wide}
\usepackage[]{graphicx}

\usepackage{hyperref} 
\theoremstyle{plain}
\newtheorem{theorem}{Theorem}

\newtheorem{proposition}{Proposition}
\newtheorem{lemma}{Lemma}

\newtheorem{remark}{Remark}
\newtheorem{corollary}{Corollary}
\newtheorem{definition}{Definition}

\numberwithin{equation}{section}

\newcommand{\vphi}{\varphi}
       
\newcommand{\fin}{\hfill$\square$}

\def \RR {{\mathbb R}}
\def \ZZ {{\mathbb Z}}
\def \NN {{\mathbb N}}

\def \cO {{\mathcal O}}
\def \cF {{\mathcal F}}
\def \cG {{\mathcal G}}

\def\cA{{\mathcal A}}

\def\cC{{\mathcal C}}
\def\cH{{\mathcal H}}

\def\-{{\setminus}}

\def\vphi{\varphi}

\begin{document}

\title[Nil-Anosov actions]
      {Nil-Anosov actions}

\author{Thierry Barbot}
\author{Carlos Maquera}
\thanks{The first author thanks the FAPESP and the French consulate in the state of S\~{a}o Paulo (Brazil) for the
financial support grant ``Chaires franco-br\'esiliennes dans l'\'etat de S\~{a}o Paulo".\\
The second author would like to thank CNPq of Brazil for financial
support Grant 200464/2007-8}



\date{\today}

\address{
Thierry Barbot\\
LMA, Avignon University\\
33 rue Louis Pasteur\\
84000 Avignon\\
France}
 \email{thierry.barbot@univ-avignon.fr}

\address{
Carlos Maquera\\
Universidade de S{\~a}o Paulo - S{\~a}o Carlos \\Instituto de
ci{\^e}ncias
matem{\'a}ticas e de Computa\c{c}{\~a}o\\
Av. do Trabalhador S{\~a}o-Carlense 400 \\
13560-970 S{\~a}o Carlos, SP\\
Brazil}

 \email{cmaquera@icmc.usp.br}

 \begin{abstract}
We consider Anosov actions of a Lie group $G$ of dimension $k$ on a closed manifold of dimension $k+n.$
We introduce the notion of Nil-Anosov action of $G$ (which includes the case where $G$ is nilpotent) and establishes the invariance by the entire group $G$
of the associated stable and unstable foliations. We then prove a spectral decomposition Theorem
for such an action when the group $G$ is nilpotent. Finally, we focus on the case where $G$ is nilpotent and
the unstable bundle has codimension one. We prove that in this case the action is a Nil-extension
over an Anosov action of an abelian Lie group. In particular:
\begin{itemize}
\item if $n \geq 3,$ then the action is topologically transitive,
\item if $n=2,$ then the action is a Nil-extension over an Anosov flow.
\end{itemize}
 \end{abstract}

 \maketitle

 \medskip
 \medskip

 \thispagestyle{empty}


 \section{Introduction}
 A locally free action $\phi$ of a group $G$ on a closed manifold $M$ is said to be \textit{Anosov} if there exists  $\alpha$ in $\mathcal{G}$, the Lie algebra of $G$, such that  $g:=\phi(\exp \alpha,\cdot)$ is normally hyperbolically with respect to the orbit foliation.
 In our serie of papers, we focus on the case of codimension one Anosov actions, i.e. Anosov actions admitting a normally hyperbolic element
 for which the unstable direction has dimension one. This notion is very classical when the
 group $G$ is $\RR,$ i.e. in the case of Anosov flows. Until recently, there have been many works establishing what can be the behavior of a codimension one
 Anosov flow in a given manifold, even aiming to describe the Anosov flow when the ambient manifold is prescribed, with emphasis in dimension $3$ (see \cite{asaoka, barbot1, barbot2, barbot3, bafe1, bafe2, BBB1, BBB2, BL, Fe1, Fe2, franks, FrWi, ghys, ghys2, ghys3, matsumoto, simic1, verjovsky}).

 In our first paper \cite{paper1}, we intend to extend this analysis in the case where $G$ is still abelian, but of bigger dimension.
 We prove in particular that if the dimension of the ambient manifold exceeds by at least $3$ the dimension of the acting abelian group,
 then the action is topologically transitive. We also proposed the following conjecture, natural extension of a previous well-known conjecture:

  \vspace{.5cm}
 \noindent
 \textbf{Generalized Verjovsky Conjecture.}
 \textit{Let $G$ be an abelian Lie group of dimension $k.$ Every codimension one faithful Anosov action of $G$ on a manifold of dimension $\geq k+3$
 is topologically conjugate the suspension of an Anosov action of $\ZZ^k$ on a closed manifold. }

 \vspace{.5cm}

 Meanwhile, V. Arakawa completely elucidated in his Ph. D Thesis the case of Anosov actions of $\RR^k$ on a closed manifold of dimension $k+2:$
 such an action reduces through a flat bundle to an Anosov flow on a $3$-manifold (\cite{arakawa}).

Afterwards, in our work \cite{algebrique} aiming to generalize Tomter's classification of \textit{algebraic Anosov flows} (\cite{tomter1, tomter2} in the context of actions of abelian groups of higher dimension, it became clear  that the natural setting was the case where the group $G$ is nilpotent. In particular,
the Generalized Verjovsky Conjecture above can be replaced by the similar statement where the group $G$ is nilpotent, not necessarily abelian. Actually,
it follows from the present paper that a nilpotent Lie group admitting a faithful Anosov action of codimension one on a closed manifold is necessarily
abelian (Theorem \ref{thm:abelian}).

We therefore started the analysis of Anosov actions of non-abelian Lie groups, which is the topic of the present paper.
The key point one needs in order to undertake such an analysis is that the Anosov action is good in the sense of \cite{tavares}, i.e. that
the Anosov splitting associated to a partially hyperbolic element
should be preserved by the other elements of $G.$ It is for this reason that previously most authors only considered the case where
the partially hyperbolic element lies in the center of $G,$
for which this condition is obviously satisfied. Here, we observe that this property still hold if $G$ is only nilpotent\footnote{In \cite{hirsch}, M.W. Hirsch already made this observation, but never afterwards published a paper including the proofs of the results stated in \cite{hirsch}.}. Actually,
we point out a more general phenomenon: in order to have the $G$-invariance of the Anosov splitting, we only need the fact that the
partially hyperbolic element belongs to the nilradical of $G:$ the Anosov action is then called \textit{Nil-Anosov} (see Theorem \ref{thm:inv_foliation}).

However, the study of Nil-Anosov actions is limited by the fact that the density of periodic orbits in the nonwandering set may fail in this more
general situation. It is for this reason that we quickly restrict ourselves to the case of nilpotent Lie groups. Nevertheless, we
consider that the study of Nil-Anosov actions is interesting in itself.

In section \ref{sec:nilpotent}, we prove a spectral decomposition Theorem for Anosov actions of nilpotent Lie groups:
the nonwandering set of such an action can be decomposed in a finite number of basic sets,
each of them being topologically transitive  (Theorem \ref{thm:espdecomp}). Observe that such a result is not merely
an adaptation of the arguments involved in the case of diffeomorphisms or flows: the notion of nonwandering set is itself
delicate and our analysis involves for example the fact that the nonwandering set of any element of $G$ coincide with
the nonwandering set of the entire group $G$ (Remark \ref{rk:omegaegal}).

One of the most important result in section \ref{sec:nilpotent} is the non-trivial fact when the Anosov action is faithful,
then the holonomy representations of each leaf of the weak foliations are faithful (Proposition \ref{le:trivial}).

This fact is the principal element of the proof of Theorem \ref{thm:abelian} mentioned above: the study of codimension one actions
of nilpotent Lie groups reduces to the case of abelian Lie groups. Therefore, the topological transitivity of these actions
in higher dimensions follows from \cite{paper1}, but here we propose an alternative simpler proof of this fact (see Theorem \ref{th:transitif}).
Finally, we conclude with an alternative and much simpler proof of Arakawa's Theorem: Anosov actions of a nilpotent Lie group of dimension $k$
on a manifold of dimension $k+2$ is a Nil-extension of an Anosov flow on a $3$-manifold.

 \vspace{.5cm}
 \textbf{Acknowledgments.} This paper was partly written while the second author stayed at Laboratoire de
de Math\'ematiques, Universit\'e d'Avignon et des Pays de Vaucluse. It has been concluded while the second author was
visiting the University of Campinas (UNICAMP) and supported by the french-brazilian program ``Chaires franco-br\'esiliennes dans l'\'etat
de S\~{a}o Paulo". Pr. R\'egis Var\~{a}o contributed to the conclusion of the paper, mainly in the elaboration of Theorem \ref{thm:3}.

\section{Preliminaries}
\label{sec.preli}

\subsection{Definitions and notations}

 Let $G$ be a connected, simply connected Lie group of dimension $k$, let $\mathcal{G}$ be the Lie algebra of $G$, and let $M$ be a $C^\infty$ manifold of dimension $n+k$,
endowed with a Riemannian metric $\|\cdot \|,$ and let $\phi$ be a locally free smooth action of the simply connected Lie group $G$ on $M$.

We will use the following usual simplified notations, denoting $\phi(g, x)$ by $g.x$ or $gx$ for every $x$ in $M$ and every $g$ in $G.$
We will however denote by $\phi^g$ the associated diffeomorphism of $M$, when non evaluated on an element of $M,$ in order to make a distinction with the element $g$ of $G.$
We denote by $G.x$ or  $\mathcal{O}_G(x):=\{gx, g \in G \}$ the orbit of
 $x \in M:$ in the second notation, we keep in mind that the orbit is the leaf of a foliation, the orbit foliation
that we denote by $\cO_G.$ We will also use the simplified notation $\cO$ and $\cO(x)$ when there is no ambiguity on
the group $G.$ We denote by $\Delta_x := \{ g \in G : gx=x \}$
the isotropy group of $x.$ The action can be thought as
 a morphism $G \rightarrow \mbox{Diff}(M)$, the kernel of which is called
  the \textit{kernel of $\phi$:} the kernel is the intersection of all the isotropy groups.
  The action  is said to be \textit{locally free} if the isotropy group of every point is
 discrete.

We can also see this data as an injective Lie morphism from $\cG$ into the Lie algebra of vector fields of $M.$
 For every $\mathfrak g$ in $\cG$, the flow $(t,x) \to \exp(t\mathfrak g)x$ will be called \textit{the flow generated by $\mathfrak g$,} denoted by $\phi_{\mathfrak g},$
and the vector field inducing this flow is denoted by $X_{\mathfrak g}.$

Let $\mathcal{F}$ be a continuous foliation on a manifold $M$. We
denote the leaf that contains $p \in M$ by $\cF(p)$. For an open
subset $U$ of $M$,
 let $\cF|_U$ be the foliation on $U$ such that
 $(\cF|_U)(p)$ is the connected component of $\cF(p) \cap U$
 containing $p \in M$.
A coordinate function $\vphi=(x_1,\cdots, x_n)$ on $U$
 is called {\it a foliation coordinate} of $\cF$
 if $x_{m+1},\cdots,x_n$ are constant functions
 on each leaf of $\cF|_U$, where $m$ is the dimension of $\cF$.
A foliation is of class $C^{r}$ if it is covered by
 $C^{r}$ foliation coordinates.
We denote the tangent bundle of $M$ by $TM$. If $\cF$ is a $C^1$
foliation, then we denote the tangent bundle
 of $\cF$ by $T\cF$.

 We fix once for all the Riemannian metric $\|\cdot \|$, and denote by
 $d$ the associated distance map in $M$. We also fix a metric $| \cdot |$
on the Lie algebra $\cG.$

 \subsection{Anosov actions}
 \label{sub.def}

Let $T\mathcal{O}$ the $k$-dimensional subbundle of $TM$
 that is tangent to the orbits of $\phi$.

 \begin{definition}\hspace{-.5cm}.
\label{defi:splitting}
 \begin{enumerate}
   \item  {\rm A \textit{$G$-splitting of $TM$} is a pair $\xi = (E_1, E_2)$ of subbundles of $TM$ such that:
\begin{equation*}
 TM=E_1 \oplus T\mathcal{O} \oplus E_2
\end{equation*}
\noindent
  }
   \item  {\rm Two $G$-splittings $\xi = (E_1, E_2)$ and $\xi' = (E'_1, E'_2)$ are \textit{transverse one to the other} if $(E_1, E'_2)$ and $(E'_1, E_2)$ are also
$G$-splittings.
   }
  \end{enumerate}
 \end{definition}

Clearly, if $\xi = (E_1, E_2)$ and $\xi' = (E'_1, E'_2)$ are transverse, then dim$E_1 =$ dim$E'_1$ and dim$E_2 =$ dim $E'_2$.
Observe that we don't require the splitting to be $G$-invariant.

 \begin{definition}\hspace{-.5cm}.
 \label{defi:defcodone}
 \begin{enumerate}
   \item  {\rm We say that $\mathfrak{a} \in \mathcal{G}$  is an \textit{Anosov element} for $\phi$ if $g=\phi^{\exp \mathfrak{a}}$
   acts normally hyperbolically with respect to the orbit foliation. That is, there exists real
  numbers $\lambda > 0,\ C > 0$ and a continuous $Dg$-invariant $G$-splitting $\xi = (E^{ss}, E^{uu})$ such that if, $m$ denotes the co-norm operator, then:

  $$
   \|Dg^n|_{E^{ss}}\|, \ \
    \|Dg^{-n}|_{E^{uu}}\|, \ \
    \frac{\|Dg^n|_{E^{ss}}\|}{m(Dg^n|_{T\mathcal{O}})},\ \
    \frac{\|{Dg^n|_{T\mathcal{O}}}\|}{m(Dg^n|_{E^{uu}})}
\leq Ce^{-\lambda n}, \ \  \textrm{for all} \ \ n\geq 0.
    $$
   In the terminology of \cite{hirpush}, it means that $g$ is $1$-normally hyperbolic to $\mathcal{O}$.
    In this case, we say that the element $\mathfrak a$ of $\mathcal G$ and the element $\exp \mathfrak a$ of $G$ are {\it $\xi$-Anosov}.
  The splitting $\xi$ is called a \textit{$g$-hyperbolic splitting.} Observe that given an Anosov element $g$, the $g$-hyperbolic splitting is unique.
  }
   \item  {\rm Call $\phi$ an \textit{Anosov action} if some $\mathfrak{a} \in \mathcal{G}$ is
   an Anosov element for $\phi$.}
   \item {\rm The action $\phi$ is a {\it codimension-one} Anosov action
 if $E^{uu}$ is one-dimensional for some Anosov element $\mathfrak{a} \in \mathcal{G}$.
   }
  \end{enumerate}
 \end{definition}

 Note that the splitting $\xi$ depends on the Anosov element $\mathfrak{a}$. By  Hirsch-Pugh-Shub theory in  \cite{hirpush}  we obtain that the $\xi$ is H\"older continuous
 and the subbundles $E^{ss}$, $E^{uu}$, $T\mathcal{O} \oplus E^{ss}$, $T\mathcal{O} \oplus E^{uu}$
 are integrable.
 The corresponding foliations, $\cF^{ss}[\xi]$, $\cF^{uu}[\xi]$, $\cF^{s}[\xi]$, $\cF^{u}[\xi]$, are called
 \textit{the strong stable foliation, the strong unstable foliation, the weak stable foliation},
 and \textit{the weak unstable foliation}, respectively. The orbits of $\phi$ are the leaves
 of a central foliation $\cO$ that we call \textit{orbit foliation}.

We insist on the fact that these foliations are not necessarily $G$-invariant: consider for example the case the
 affine group acting on the left on the quotient SL$(2, \mathbb R)/\Gamma$
where $\Gamma$ is a cocompact lattice of SL$(2, \mathbb R)$ and $G$ the group Aff of upper triangular matrices. In this case, any non-trivial diagonal matrix is an Anosov element $g$ of Aff for which,
let's say, the unstable bundle $E^{uu}$ is trivial. Then the stable bundle $E^{ss}$ is $g$-invariant, but not Aff-invariant.
%

   \begin{remark}
   \label{rmk:convex_cone}
      {\rm
    For any splitting $\xi$, let $\mathcal{A}_\xi$ denote the set of $\xi$-Anosov elements. By the Lie product formula we have, for any $\mathfrak a$, $\mathfrak b$ in $\mathcal A_\xi:$
   $$
   \exp(s\mathfrak{a}+t\mathfrak{b})=\lim_{n\to \infty} \big[\exp(\frac{s\mathfrak{a}}{n}) \exp(\frac{t\mathfrak{b}}{n})\big]^n, \ \ \textrm{for any positive real numbers}\ \ s, t.
   $$
Consequently, $s\mathfrak{a}+t\mathfrak{b}$ is an Anosov element in $\mathcal{A}_\xi$ which preserves $\xi$.  It follows that $\mathcal A_\xi$ is a convex cone in $\mathcal G$.
  }
 \end{remark}

 \begin{definition}[Anosov subcone]
 \label{def:subcone}
{\rm
A \textit{subcone} $\cC$ is an open convex cone in $\mathcal G.$ A subcone is \textit{Anosov} if it is made of Anosov elements. It is \textit{strict} if any nonzero element of the closure $\overline{\cC}$ is Anosov.

The \textit{$\cC$-orbit of a point $x,$} denoted by $\cC_x,$ is the set comprising points of the form $\exp(\mathfrak g_1)...\exp(\mathfrak g_N)x$ where the $\mathfrak g_i$'s are elements of $\cC$ in arbitrary numbers.
}
\end{definition}

We consider Anosov subcones as approximations of one parameter subgroups: the $\cC$-orbit of a point $x$ is contained in the orbit $\cO_x$ and
is made of the points of the form $c(1)$ where $c: [0, 1] \to M$ is a differentiable curve  such that $c(0)=x$ and  at any time $t$, the derivative $\dot{c}(t)$ is tangent to the orbit of the flow $\phi_{\mathfrak g}$ where $\mathfrak g$ is an element of $\cC.$ The thiner the cone is, the closer are the $\cC$-orbits to the trajectory of $\phi_{\mathfrak g}$ for any element $\mathfrak g$ of $\cG.$

   The following remark is a consequence of the structural stability of Anosov elements (see  Theorems (7.1), (7.3) and the Remark after Theorem (7.3) in  \cite{hirpush}).

 \begin{remark}
 \label{rk.first}
 The set  $\mathcal{A}=\mathcal{A(\phi)}$ of Anosov elements for $\phi$
  is an open subset of $\mathcal{G}$. More precisely, the map associating to an element $\mathfrak a$ in $\mathcal A$ its Anosov splitting $\xi(\mathfrak a)$ is continuous.
  \end{remark}

 \begin{remark}
 \label{rk.transversesplit}
On the other hand, there is a neighborhood $U$ of $\xi(\mathfrak a)$ in the space of splittings (with prescribed stable/unstable dimensions) such that $\xi(\mathfrak a)$ is the only $\mathfrak a$-invariant splitting in $U$.

More precisely: let $\xi' = (E^1, E^2)$ be any $G$-splitting transverse to $\xi = (E^{ss}, E^{uu})$. Then, at every point $x$ of $M$, the subspace $E^1_x$ is the graph of a linear endomorphism $E_x^{ss} \to E^{uu}_x \oplus T\cO_x$, hence $\xi'$ provides a section $\sigma$ of Hom$(E^{ss}, E^{uu} \oplus T\cO)$. If $\xi'$ is $\exp(\mathfrak a)$-invariant, then $\sigma$ is $\exp(\mathfrak a)$-equivariant. But it follows easily from the Anosov property that the unique  $\exp(\mathfrak a)$-equivariant section of Hom$(E^{ss}, E^{uu} \oplus T\cO)$ is the null one. Therefore we have $E^1 = E^{ss},$ and, similarly, $E^2 = E^{uu}$.
  \end{remark}

  \subsection{Nil-Anosov actions}

    Assume that $G$ has a nontrivial nilradical $N.$
We call $\phi$ a \textit{Nil-Anosov action} if some $\mathfrak{a}\in \mathcal{N},$ the Lie algebra of $N,$ is an Anosov element for $\phi$. In this case $\mathfrak{a}$ is call a \textit{Nil-Anosov element} for $\phi.$

   \begin{theorem}
   \label{thm:inv_foliation}
 If $\phi$ is a Nil-Anosov action, then for any Anosov element $\mathfrak a$ of the nilradical $N$, the hyperbolic splitting $\xi$ for which
  $\mathfrak a$ is Anosov is preserved by the action of $G.$ In particular, every foliation $\cF^{ss}[\xi]$, $\cF^{uu}[\xi]$, $\cF^{s}[\xi]$, $\cF^{u}[\xi]$ is
 $G$-invariant.
 \end{theorem}

 The proof of this Theorem is based on this fundamental lemma, that we will use several time in this paper:

\begin{lemma}\textbf{(Ascending chains of normalizers stabilize)}
\label{le:ascendingchain}
Let $\cH_0$ be a Lie subalgebra in a nilpotent Lie algebra  $\mathcal N$. Let $(\cH_i)_{(i \in \NN)}$ be the infinite sequence obtained by taking the successive normalizers:
for every integer $i$,  $\cH_{i+1}$ is the normalizer of $\cH_i$ in $\mathcal N$. Then,
this sequence eventually stabilizes to $\mathcal N$: there is an integer $k$ such that $\cH_k = \mathcal N$.
\end{lemma}

\begin{proof}
Nilpotent groups satisfies the \textit{normalizer condition:} for every Lie subalgebra $\mathcal H$, the normalizer $N(\mathcal H)$  contains strictly $\mathcal H$, unless
$\mathcal H$ is already the entire Lie algebra $\mathcal N$. It follows that the dimension of $\mathcal H_i$ increases with $i$, until the equality $\mathcal H_k = \mathcal N$ holds.
\end{proof}

 \begin{lemma}
 \label{lem:normalizer}
 Let $\cH$ be a Lie subalgebra of $\cG$ containing a $\xi$-Anosov element and assume that for every $\mathfrak a$ in $\cH$ the splitting $\xi$ is $\exp(\mathfrak a)$-invariant.
Then, for every $\mathfrak h$ in the normalizer $N_\cG(\cH)$  in $\cG$ the splitting $\xi$ is $\exp(\mathfrak h)$ invariant.
 \end{lemma}

  \begin{proof}
  Let $a = \exp(\mathfrak a)$ be an Anosov element in $H = \exp(\cH).$ Then, for all $\mathfrak h$ in $N_\cG(\cH)$ the conjugate $\exp(\mathfrak h)a\exp(-\mathfrak h)$
  is $\exp(\mathfrak h)(\xi)$-Anosov. If $\mathfrak h$ is small, the splittings $\xi$ and $\exp(\mathfrak h)(\xi)$ are close one to the other, in particular, transverse one to the other.
Since by hypothesis $\xi$ is $\exp(\mathfrak h)a\exp(-\mathfrak h)$-invariant (since $\exp(\mathfrak h)a\exp(-\mathfrak h)$ lies in $\cH$), it follows from Remark \ref{rk.transversesplit} that we have:
  $$
  \exp(\mathfrak h)(\xi)=\xi,
  $$
 The lemma follows since every element of $\exp(N_\cG(\cH))$ is an iterate of some $ \exp(\mathfrak h)$ with  $\mathfrak h$ small.
  \end{proof}

   \begin{proof}[Proof of Theorem \ref{thm:inv_foliation}]
   Let $\cH_0$ be the subalgebra of $\mathcal N$ generated by the Nil-Anosov element $\mathfrak a.$ We define, inductively, $\cH_{i+1}=N_{\mathcal N}(\cH_i)$ for all $i$ in $\mathbb{N}.$
By Lemma \ref{lem:normalizer} we have that $\xi$ is preserved by $\cH_i$ for every $i$ in $\mathbb{N}.$ On the other hand, by Lemma \ref{le:ascendingchain}, there exists $k$ in $\mathbb{N}$ such that $\mathcal N=\cH_k.$

 Finally, by applying Lemma \ref{lem:normalizer} to $N_\cG(\mathcal N)$ and
using the fact that $N_\cG(\mathcal N)=\cG$, because $\mathcal N$ is the nilradical of $\cG$,
we obtain that $G$ preserves $\xi.$
 \end{proof}

\begin{proposition}
\label{pro:opencone}
Let $\mathfrak a_0$ be an element of $\mathcal N$ which is $\xi$-Anosov. Then, the set $\mathcal A_\xi$ of $\xi$-Anosov elements of $G$ (not necessarily in the nilradical)
is an open convex cone, which is a connected component of the set $\mathcal A$ of Anosov elements of $G$.
\end{proposition}

\begin{proof}
By Remark \ref{rmk:convex_cone}, $\mathcal A_\xi$ is a convex cone.
According to Remark \ref{rk.first}, $\mathcal A$ is an open neighborhood of $\mathcal A_\xi.$
The map associating to an element $\mathfrak a$ in $\mathcal A$ its Anosov splitting $\xi(\mathfrak a)$ is continuous; and it follows from Theorem \ref{thm:inv_foliation} and Remark \ref{rk.transversesplit} that
it is locally constant at every point in $\cA_\xi$. It follows that $\cA_\xi$ is open in $\cA$. The proposition follows since $\cA_\xi$ is clearly closed in $\cA.$
\end{proof}

Connected components of $\mathcal{A}$ are called \textit{chambers}.
From now we will assume that every Nil-Anosov action comes with a specified Nil-Anosov  element $\mathfrak{a}_0$.
We denote by $\xi$ the splitting for which $\mathfrak a_0$ is $\xi$-Anosov, and by $\cA_0 = \cA_\xi$ the chamber containing the preferred Anosov element $\mathfrak{a}_0$.
The splitting $\xi$ defines the stable/unstable foliations common to every element in $\mathcal{A}_0$, that we will denote by $\mathcal{F}^s$, $\mathcal{F}^u$, $\mathcal{F}^{ss}$ and $\mathcal{F}^{uu},$
dropping the notation $[\xi].$ According to Theorem \ref{thm:inv_foliation}, all these foliations are $G$-invariant.


 For all $\delta >0$, $\cF_{\delta}^{i}(x)$
 denotes the open ball in $\cF^i(x)$ centered at $x$ with radius
 $\delta$ with respect to the restriction of $\|\cdot \|$ to $\cF^i(x)$, where $i=ss,uu,s,u.$

 \begin{theorem}[local product structure]
 \label{thm:local product}
 Let $\phi:G \times M \to M$ be a Nil-Anosov action. There exists a
 $\delta_0 >0$ such that for all $\delta \in (0,\delta_0)$ and for
 all $x\in M,$ the applications
 $$
 [\cdot,\cdot]^u:\cF^s_\delta(x) \times \cF^{uu}_\delta(x)\to M; \ \
 [y,z]^u=\cF^s_{2\delta}(z)\cap \cF^{uu}_{2\delta}(y)
 $$
 $$
 [\cdot,\cdot]^s:\cF^{ss}_\delta(x)\times \cF^{u}_\delta(x)\to M; \ \
 [y,z]^s=\cF^{ss}_{2\delta}(z)\cap \cF^{u}_{2\delta}(y)
 $$

 are homeomorphisms onto their images.
 \end{theorem}

 \begin{remark}\hspace{-.2cm}
 \label{rk.zero}
 \begin{enumerate}
 \item  {\rm
Every leaf of $\cF^{ss}$ or $\cF^{uu}$ is a plane, i.e. diffeomorphic to $\RR^\ell$ for some $\ell   $.
  This is a straightforward observation, since every compact domain in a leaf
 of $\cF^{ss}$ (respectively
 of $\cF^{uu}$) shrinks to a point under positive (respectively negative) iteration by $\phi^{\exp \mathfrak{a}_0}$.
 \item Let $F$ be a weak leaf, let us say a weak stable leaf. For every strong stable leaf $L$ in $F$, let
 $\Gamma_L$ be the subgroup of $G$ comprising elements $a$ such that $\phi^a(L)=L$,
 and let $\omega_L$ be the saturation of $L$ under $\phi$. Thanks to Theorem~\ref{thm:local product} we have:
 \begin{itemize}
 \item $\omega_L$ is open in $F$,
 \item $\Gamma_L$ is discrete.
 \end{itemize}
 Since $F$ is connected, the first item implies $F=\omega_L$: the $\phi$-saturation of a strong leaf  is an entire weak leaf.
 The second item implies that the quotient $P=G/\Gamma_L$ is a manifold. For every $x$ in $F$, define $p_F(x)$ as the equivalence class $a\Gamma_L$ such that
 $x$ belongs to $\phi^a(L)$.  The map $p_F: F \to P$ is a locally trivial fibration
 and the restriction of $p_F$ to any $\phi$-orbit in $F$ is a covering map.
 Since the fibers are contractible (they are leaves of $\cF^{ss}$, hence planes), the fundamental group
 of $F$ is the fundamental group of $P$, i.e. $\Gamma_L$ for any strong stable leaf $L$ inside $F$.

 Observe that if $\cF^{ss}$ is oriented, then the fibration $p_F$ is trivial: in particular, $F$ is diffeomorphic to
 $P \times \RR^{p}$, where $p$ is the dimension of $\cF^{ss}$.
 Of course, analogous statements hold for the strong and weak unstable leaves.
 }
\item {\rm
In the continuation of the previous items, let $F$ be a weak stable leaf, $x$ a point in $F$ and $L$ the strong stable leaf through $x$. Since $p_F: F \to P$
is a locally trivial fibration and the fibers $L$ are contractible, loops in $F$ based at $x$ are all homotopic to a loop $\omega_x^{\mathfrak h}$ obtained by composing a path from $[0,1]$ into
$\cO_x$ of the form $t \mapsto \exp(t\mathfrak h)$ where $\mathfrak h \in \cG$ satisfies $(\exp\mathfrak h)L=L$ with a path joining $\exp(\mathfrak h).x$ to $x.$
The holonomy of such a loop does not depend on the choice of the final path joining $\exp(\mathfrak h).x$ to $x$ since $L$ is contractible; this holonomy
defines a germ at $x$ of homeomorphism of $\cF^{uu}(x)$ which only depends on $x$ and $\mathfrak h.$ We denote it by $h_x^{\mathfrak a}.$
}
 \end{enumerate}
 \end{remark}

   \begin{remark}
   \label{rk:periodiclocal}
   {\rm
   We will also need a local description of fixed points of Anosov elements:
   let $a = \exp(\mathfrak a)$ be an $\xi$-Anosov element of an Anosov action of $G$. Let $x$ be a fixed point
   of $\phi^a$. Let $D$ be a small disk  containing $x$ transverse to the orbit foliation $\cO$.
   There is a small neighborhood $U$ of
   $x$ in $M$ and a neighborhood $V$ of $e$ (the identity element) in $G$ such that, for every $z$ in $U$, there is one and only one element $b$ in $V$ such that $bz$ belongs to $D$. Moreover,
   reducing $U$ if necessary, we can assume that for every $b$ in $V$ the product $ba$ is Anosov.
   Let $D' \subset D$ be a subdisk such that $\phi^{a}(D')$ is contained in $U$. Then there is a ``first return map'' $\varphi: D' \rightarrow D$ and a ``variation of time of first return map''
   $\alpha: D' \rightarrow V$ uniquely defined by:
   $\varphi(z)={\alpha(z)a}.z \in D$. Observe that the map $\varphi$ could also be defined
   as the holonomy of $\cO$ along the periodic orbit of the flow $\phi_{\mathfrak a}$
   containing $x$.

   The weak foliations $\cF^{s}$ and $\cF^u$ are transverse to $D'$ and thus define two foliations
   $\cG^s$ and $\cG^u$, transverse to one another, of supplementary dimensions, and preserved by the holonomy map $\varphi$. Let $s_0$ (respectively $u_0$) the leaf of $\cG^s$ (resp. $\cG^u$)
   through $x$. They are both $\varphi$-invariant, the action of $\varphi$ on $s_0$ is contracting, and the action on $u_0$ is expanding. By considering the map from $D'$ into $s_0 \times u_0$ associating to a point
   $y$ the pair $(p, q)$, where $p$ is the intersection between $\cG^u(y)$ and $s_0$, and $q$ the intersection between $\cG^s(y)$ and $u_0$ (cf. Theorem~\ref{thm:local product}), we get that $x$ is a saddle fixed point of $\varphi$. In particular, $x$ is an isolated fixed point of $\varphi$.
   }
   \end{remark}

 \begin{remark}
 \label{rmk:lattice_anosov}
 If $H$ is a Lie subgroup of $N$ containing a Nil-Anosov element, then
 every uniform lattice in $H$ contains Anosov elements. In particular, the isotropy subgroup of a point in $M$ whose $H$-orbit is compact contains Anosov elements.
 {\rm
 Indeed, if $\Delta$ is an uniform lattice of $H$, then, by a result in \cite[Theorem (5.1)]{loglattice}, $\Delta$ contains a log-lattice, i.e. a finite index sub-lattice $\Delta' \subset \Delta$ such that $\log \Delta'$
is a lattice of the vector space $\mathcal{H}$. On the other hand, let $\mathfrak{h}$ be a Nil-Anosov element in $\mathcal{H}$ and let $\mathcal{A}_0$ be the chamber in $\mathcal{G}$ containing the Anosov element $\mathfrak{h}$. Then, since Proposition \ref{pro:opencone} implies that $\mathcal{H}\cap \mathcal{A}_0$ is an open convex cone in $\mathcal{H},$ the set
$$
(\mathcal{H}\cap \mathcal{A}_0)\cap \log \Delta'
$$ is not empty. Consequently, $\Delta'\subset \Delta$ contains Anosov elements, and thus, our assertion follows.
}
 \end{remark}

\subsection{Examples}
\label{sec:exemple}

In this subsection, we give examples of general Anosov actions, not necessarily Nil-Anosov. We start showing how to produce new Anosov actions from the data of Anosov actions.

\subsubsection{Products of Anosov actions} If  $\phi_1: G_1 \times M_1 \to M_1$ and  $\phi_2: G_2 \times M_2 \to M_2$ are two Anosov actions,
then the product action of $G_1 \times G_2$ on $M_1 \times M_2$ is clearly Anosov.

The nilradical of $G_1 \times G_2$ is the product of the nilradicals of $G_1$ and $G_2,$ and an element $(g_1, g_2)$ of $G_1 \times G_2$ is Anosov for $\phi_1 \times \phi_2$ if and only if $g_1$ is Anosov for $\phi_1$ and $g_2$ is Anosov for $\phi_2.$ It follows that the product action is Nil-Anosov if and only if
$\phi_1: G_1 \times M_1 \to M_1$ and  $\phi_2: G_2 \times M_2 \to M_2$ are both Nil-Anosov.

This remark includes the limit case where one of the action, let's say $\phi_2: G_2 \times M_2 \to M_2$, is transitive, i.e. an action of $G_2$ by left translations on
a compact quotient $G_2/\Lambda_2.$  This is partially generalized in the next section.

\subsubsection{Nil-extensions}
\label{sub:central}
 We consider any extension:
$$0 \to H \to \widehat{G}   \overset{p}\rightarrow G \to 0$$
where $H$ is a nilpotent Lie group. A locally free action $\hat{\phi}: \widehat{G} \times \widehat{M} \to \widehat{M}$ is a \textit{Nil-extension} of $\phi: G \times M \to M$
if there is a $\widehat{G}$-equivariant map $\pi: \widehat{M} \to M$ whose fibers are precisely the $H$-orbits:
$$ \pi(\hat{g}.x) = p(\hat{g}).\pi(x).$$
We furthermore require that the map $x \mapsto H_x$ where $H_x$ is the stabilizer of $x$ for the action of $H$ is continuous as map from
$\widehat M$ into the space of lattices of $H.$

In other words, $\pi$ is a locally trivial fibration over $M,$ with fibers homeomorphic to compact quotients of $H.$
It induces an identification between the orbit space $H\backslash\widehat{M}$ and $M$, and $\phi$ is the induced action.

A Nil-extension introduces no additional dynamical feature. In particular, elements of $H$ have a ``trivial'' transversal action;
they cannot be Anosov relatively to the action $\hat{\phi}.$ Hence, clearly, $\phi$ is (Nil-)Anosov if and only if $\hat{\phi}$ is (Nil-)Anosov.

A fundamental type of Nil-extensions are the \textit{central extensions,} i.e. Nil-extensions for which $N$ is in the center of $G.$
Observe that central extensions are not necessarily product; see Starkov's example described  in \cite[Section $3.2.3$]{algebrique}.

\subsubsection{Affine group of the real line}
Let Aff be the subgroup of SL$(2, \RR)$ made of upper diagonal matrices: we have already considered previously the
action of Aff on compact quotients SL$(2, \RR)/\Gamma$ as an example of Anosov action for which the splitting of Anosov elements is not preserved by the entire group.

This Anosov action is not Nil-Anosov. Observe that it admits no compact orbits.

\subsubsection{Engel subalgebras}
Let $G$ be a real Lie group, not necessarily semisimple, and let $\cG$ be its Lie algebra. For any element $\mathfrak h$ of
$\cG$ we denote by ad $\mathfrak h$ the adjoint action of $\mathfrak h$ on $\cG.$ We denote by $L_0(\mathfrak h)$
the $0$-characteristic subspace of  ad $\mathfrak h:$ it is a subalgebra of $\cG$, called a \textit{Engel subalgebra.}
We actually have a splitting of $\cG$ by characteristic subspaces:
$$\cG = L_0(\mathfrak h) \oplus L^s(\mathfrak h) \oplus L^u(\mathfrak h) \oplus L^i(\mathfrak h)$$
where $L^s(\mathfrak h)$ (resp. $L^u(\mathfrak h)$) is the sum of characteristic subspaces associated to eigenvalues
with negative real part (resp. positive real part), and where $L^i(\mathfrak h)$ is the sum of characteristic subspaces associated
to purely imaginary eigenvalues.

The sum $\mathcal H = L_0(\mathfrak h) \oplus L^i(\mathfrak h)$ is a subalgebra. It follows immediatly that, for any cocompact lattices $\Gamma$ of $G$ the right action on $\Gamma\backslash G$
of the Lie group $H$ admitting $\cH$ as Lie algebra is Anosov.

Moreover, if $L_0(\mathfrak h)$ is \textit{minimal} (i.e. does not contain a proper Engel subalgebra) then it is a \textit{Cartan subalgebra} (abrev. CSA) i.e. is nilpotent and equal to its own normalizer (see \cite[Theorem page $80$]{humphreys}).
It happens - for example if $G$ is real semisimple and $L_0(\mathfrak h)$ a hyperbolic CSA, see \cite[Appendix $A.3$]{algebrique} - that $L^i(\mathfrak h)$ is trivial, then we obtain in this case an Anosov action of a nilpotent Lie group.

More generally, one can construct from CSA's many examples of \textit{algebraic Anosov actions,} i.e. actions on quotients
$\Gamma\backslash G/K$ of Lie subalgebras $\cH$ of $\cG$ where $K$ is a compact subgroup commuting with $\exp(\cH).$
In \cite{algebrique}, the authors study in detail these examples when $\cH$ is nilpotent.

\subsection{Nil-Anosov actions are Nil-faithful up to Nil-extensions}

We can reduce any action $\phi: G \times M \to M$ to a faithful one simply by replacing $G$ by $G/\ker\phi$ and considering the induced action on $M.$ On the other hand, many aspects of nilpotent Lie groups are simplified when they are simply connected: the exponential map is a diffeomorphism, and one can more easily manipulate subgroups of $N$ when it is simply connected.
By taking the universal covering if necessary, we can easily reduce to the case where $G$ and its nilradical are simply connected, but then we may loose the faithfulness of the action.

Nevertheless, the following proposition shows that, up to Nil-extensions, we can keep the  fact that the nilradical is simply connected.

\begin{definition}
Let $\phi: G \times M \to M$ be an action of a simply connected Lie group $G.$ The action is \textit{Nil-faithful} if the induced action of $N$ is faithful.
\end{definition}

\begin{lemma}
\label{le:centralextension}
Every Nil-Anosov action of a simply connected Lie group is a Nil-extension over a Nil-faithful Nil-Anosov action of a simply connected Lie group.
\end{lemma}

\begin{proof}
Let $\phi: G \times M \to M$ be a  Nil-Anosov action which is not Nil-faithful. Let $\Lambda$ be the intersection between the nilradical $N$ and the kernel of $\phi.$ Then $\Lambda$ is discrete in $N$ and stable by conjugacy: it follows that $\Lambda$ is in the centralizer $Z(G)$ of $G$. Since $N$ is nilpotent and simply connected, $Z(N) = Z(G) \cap N$ is isomorphic to $\RR^m$ for some $m;$ let
$H$ be the minimal connected Lie subgroup of $Z(N)$ containing $\Lambda.$ Then the torus $T := \Lambda\backslash H$ acts freely and properly on $M,$ and the orbit space $\overline{M} := T\backslash M$ is equipped with an Anosov action of  $H\backslash G$ whose nilradical $H\backslash N$ is simply connected.
The quotient map $M \to \overline{M}$ is a central extension.

The Lemma is obtained by repeating the argument until the action on the quotient space is Nil-faithful.
\end{proof}

\section{Anosov actions of nilpotent Lie groups}\label{sec:nilpotent}
From now on, we restrict to the case where the Lie group $G$ is nilpotent, i.e. equal to its nilradical $N$.
In this section, we will observe that compact orbits play a central role in the dynamical behavior of Anosov actions of nilpotent Lie groups.

\medskip

\begin{center}
\textbf{Convention:} We denote by Comp$(\phi)$ the set of compact $\phi$-orbits.
\end{center}
\medskip

By Lemma \ref{le:centralextension}, we can assume without loss of generality that $N = G$ is simply connected (all our results are insensitive to Nil-extensions).

\subsection{Compact orbits}

 Most of the content of this section and of the following section \ref{sec:spectral} is announced in \cite{hirsch}, but the proofs have not been published anywhere.

\begin{theorem}
\label{thm:compact_orbit}
Let $\phi$ be an Anosov action of a nilpotent Lie group $N$. A $\phi$-orbit is compact if and only if it contains a point fixed by an Anosov element.
\end{theorem}

The Theorem will be proved by induction involving the following Lemma.

 \begin{lemma}
 \label{lem:orbit_normalizer}
 Let $H$ be a subgroup of $N$ containing an Anosov element $a=\exp(\mathfrak a)$ and  let $x_0$ be an element of $M.$ Let $N(H)$ denote the normalizer in $N$ of $H.$ If  the $H$-orbit of  $x_0$ is compact, then
the $N(H)$-orbit of $x_0$ is compact.
 \end{lemma}
 \begin{proof}
 Assume that $\mathcal{A}_{0}$ is the chamber containing the Anosov element $\mathfrak a$ in $\mathcal{H}$. Let $\mathcal{O}_H(x_0)$ be the $H$-orbit of $x_0$, and  let $\Delta_0$ be the isotropy subgroup of $x_0$ by the $H$-action, that is $\Delta_0$ is the uniform lattice in $H$ which is given by $\Delta_0=\{h\in H:\ \phi^{h}(x_0)=x_0\}.$ Let $x_{\infty}$ be a point in the closure of the $N(H)$-orbit of $x_0,$ then there exists a sequence of elements $z_n$ of $N(H)$ such that
 $$
 \lim_{n\to \infty}\phi^{z_n}(x_0)=x_{\infty}.
 $$
 As  $z_nH=Hz_n$, then the $H$-orbit of $\phi^{z_n}(x_0)$ is exactly equal to $\phi^{z_n}(\mathcal{O}_H(x_0))$, in particular, it is compact. Hence, for every $n$ the isotropy subgroup of $\phi^{z_n}(x_0)$ by the $H$-action, which we denote by $\Delta_n$, is a uniform lattice in $H.$  Moreover, for every $n$ we have
 $$
 \Delta_n=z_n\Delta_0z_n^{-1}.
 $$

Since $N$ is nilpotent, the adjoint action of $z_n$ on $H$ is unimodular.
In particular, all the lattices $\Delta_n$ have the same co-area in $H$. Moreover, since the action on the compact manifold $M$ is locally free, there is a uniform bound from below for the size of elements
in $\Delta_n$.
 By Mahler's criterion for nilpotent Lie groups with a $\mathbb Q$-structure (which is the case for $H$ since it contains a cocompact lattice) (see \cite{mahler}), up to a subsequence, we can suppose that the sequence $\Delta_n$ converges to some uniform lattice $\Delta_{\infty}$ in $H$.  In particular, for every $h_{\infty}$ in $\Delta_{\infty},$ there is a sequence of elements  $h_n$ of $\Delta_n$ converging in $H$ to $h_{\infty}.$  Let $\gamma_n$ be the element in $\Delta_0$ such that $h_n=z_n\gamma_nz_n^{-1}$. Then, by continuity of the map $\phi: N \times M \to M$:

 $$
 \begin{array}{ccl}
h_{\infty}.x_{\infty}& = & \lim_{n\to \infty} h_n.({z_n}.x_0)\\
 & =& \lim_{n\to \infty} z_n\gamma_nz_n ^{-1}z_n. x_0\\
& =& \lim_{n\to \infty} z_n\gamma_n.x_0\\
 &=& \lim_{n\to \infty} z_n.x_0\\
 & = & x_{\infty}.
 \end{array}
 $$
 This shows that $\Delta_{\infty}$ lies in the isotropy subgroup of $x_{\infty}$ by the $H$-action and, consequently, the  $H$-orbit of $x_{\infty}$ is compact.

 We claim that
 $x_{\infty}$ is in the $\phi$-orbit of $x_0$.
 In fact, by Remark \ref{rmk:lattice_anosov}, there is  an Anosov element $a$ in $H$ such that: $\phi^{a}(x_{\infty})=x_{\infty}$ (in other words, $a$ is in $\Delta_{\infty}$) and $\log a$ is an element of the chamber $\mathcal{A}_0$.
 As defined in Remark~\ref{rk:periodiclocal},  we consider
   the local return map $\varphi: D' \rightarrow D$ along the $\phi^a$-orbit of $x_\infty$, where $D'$ and $D$ are small disks transverse to
   the $\phi$-orbit of $x_{\infty}$, and $x_{\infty}$ is the unique fixed point for $\varphi$, of hyperbolic type.

Let $V_H$ be a small neighborhood of the identity in $H$ such that the map $h \mapsto hx$ from $V_H$ to $M$ is injective for every element $x$ of $D$.
When $x$ describes $D'$, the pieces of orbits $V_H.x$  are disjoint one from the other, and have an area uniformly bounded from below.
Since the orbits $\cO_H(z_n.x_0)$ have all the same area, it follows that the number of intersections between $D$
and every  $\cO_H(z_n.x_0)$ is bounded from above, uniformly in $n$, by some integer $m$.

On the other hand, the points $\phi^{z_n}(x_0)$ accumulate on $x_{\infty}$, hence belong to $D'$ for $n$ sufficiently big.
Then, if $n$ is sufficiently big,  the iterates $\varphi^k(\phi^{z_n}(x_0))$ belongs to $D$ for all intergers $k$ between $0$ and $m+1$. By definition of $m$,
two of these iterates must be equal, but  since $x_\infty$ is a fixed point of $\varphi$ of hyperbolic type,  it is possible only if  $\phi^{z_n}(x_0) = x_\infty$.
Hence, $x_{\infty}$ belongs  to the $\phi$-orbit of $x_0$, and thus, our claim is verified.

 Let $g$ be an element of $N$ such that $x_{\infty}={g}.x_0$. This implies that $\Delta_{\infty}=g\Delta_0g^{-1}$, hence $g$ is an element of $N(H)$ since $\Delta_\infty$ and $\Delta_0$ are both lattices in $H$.

Thus, we have proved that every point $x_\infty$ in the closure of the $N(H)$-orbit of $x_0$ is in the $N(H)$-orbit of $x_0$: this orbit is compact.

 \end{proof}

  \begin{proof}[Proof of Theorem \ref{thm:compact_orbit}]
One of the implication is a direct consequence of Remark \ref{rmk:lattice_anosov}. Let us prove the other implication: let $x_0$ be  an element of $M$ fixed by an Anosov  element $a$ in $N.$ We will show that the $\phi$-orbit of $x_0$ is compact.
 In fact, let $\mathfrak a$ be the element in $\mathcal{N}$ such that $\exp \mathfrak a=a$. We consider $H_0$, the subgroup of $N$ which is the one-parameter group associated to  $\mathfrak a.$ It is clear that the $H_0$-orbit of $x_0$ is compact (a circle). We define, inductively, $H_{i+1}=N(H_i)$ for all $i$ in $\mathbb{N}.$ By Lemma \ref{lem:orbit_normalizer} we have that the $H_i$-orbit of $x_0$ is compact for every $i$ in $\mathbb{N}.$ Hence, by Lemma \ref{le:ascendingchain}, there exists $k$ in $\mathbb{N}$ such that $N=H_k$ and the Theorem follows.
 \end{proof}

Let us fix a Haar measure on $N$, so that one can estimate the area of every $\phi$-orbit.

\begin{proposition}
\label{le:raghu}
For every $C>0$, there is only a finite number of compact $\phi$-orbits of area $\leq C$.
\end{proposition}

\begin{proof}
Assume by contradiction the existence of an infinite sequence of distinct compact orbits $\cO_n$
of area $\leq C$. For each of them, let $\Delta_n$ be the isotropy group of $\cO_n$:
it is a lattice in $N$.

Since $\phi$ is locally free, the length of elements of $\Delta_n$ is uniformly bounded from below,
independantly from $n$. Once more, by the Mahler's criterion extended to nilpotent Lie groups (\cite{mahler}), it ensures that, up to a subsequence,
the $\Delta_n$ converges to some lattice $\Delta_\infty$. In particular, for every $a_\infty$
in $\Delta_\infty$, there is a sequence
of elements $a_n$ of $\Delta_n$ converging in $N$ to $a_\infty$. Furthermore, according
to Remark \ref{rmk:lattice_anosov}, we can select $a_\infty$ to be Anosov.
Up to a subsequence, we can also pick up a sequence of elements $x_n$ in each
$\cO_n$ converging to some $x_\infty$ in $M$.
Then, since $\phi^{a_n}(x_n)=x_n$, at the limit we have $\phi^{a_\infty}(x_\infty)=x_\infty$.
Since $a_\infty$ is Anosov, the $\phi$-orbit $\cO_\infty$ of $x_\infty$ is compact.
Consider a local section $\Sigma$ to $\phi$ containing $x_\infty$: the first return map on $\Sigma$
along the orbit of $\phi^{a_\infty}$ is hyperbolic, admitting $x_\infty$ as an isolated fixed point.
On the other hand, by pushing slightly along $\phi$, we can assume without loss of generality
that every $x_n$ belongs to $\Sigma$. Since the $a_n$ converges to $a_\infty$, the $\phi^{a_n}$-orbit
of $x_n$ approximates the $\phi^{a_\infty}$-orbit of $x_\infty$, showing that the $x_n$ are also
fixed points of the first return map. It is a contradiction, since they accumulate to the isolated
fixed point $x_\infty$.
\end{proof}

\subsection{Spectral decomposition}
\label{sec:spectral}

 \begin{definition}[Nonwandering set of an Anosov element]
 \label{def:nonwandering}
 {\rm
 A point $x\in M$ is \textit{nonwandering} with respect to an Anosov element $\mathfrak a$
 if for any open set $U$ containing $x$ there is a real number $t>1,$
 such that $\phi^{\exp(t\mathfrak a)}(U)\cap U \neq \emptyset.$
 The set of all nonwandering points, with respect to $\mathfrak a$, is denoted by
 $\Omega(\mathfrak a)$.
 }
 \end{definition}

Recall that we have fixed a metric $| \cdot |$ on $\mathcal N = \cG,$ and than an Anosov subcone $\cC$ is strict if nonzero elements of $\overline{\cC}$ are Anosov (Definition \ref{def:subcone}).
 \begin{definition}[Nonwandering set of a strict Anosov subcone]
 \label{def:nonwandering}
 {\rm
 A point $x\in M$ is \textit{nonwandering} with respect to a strict Anosov subcone $\cC$
 if for any open set $U$ containing $x$ there is an element $\mathfrak a$ of $\cC$ such that $| \mathfrak a | \geq 1$ and $\phi^{\exp(\mathfrak a)}(U)\cap U \neq \emptyset.$
 The set of all nonwandering points, with respect to $\cC$, is denoted by
 $\Omega(\cC)$.
 }
 \end{definition}

Let $\cC$ be a strict Anosov subcone containing $\mathfrak a$. Observe that $\Omega(\mathfrak a)$ and $\Omega(\cC)$ are closed, and
the inclusion $\Omega(\mathfrak a) \subset \Omega(\cC)$ is obvious.
If the $\phi$-orbit of a point $x$ is compact, then the restriction of the action of $\mathfrak a$ on this orbit is a left translation on
a compact right quotient of $N$ by a lattice; hence this action preserves a volume form on the orbit and every point is nonwandering. The inclusions Comp$(\Phi) \subset \Omega(\mathfrak a) \subset \Omega(\cC)$ follow.
Inversely:


  \begin{proposition}
 \label{prop:closCompactO}
  For any strict Anosov subcone $\cC$, the union Comp$(\phi)$ of compact orbits of $N$ is dense in $\Omega(\cC)$.
 \end{proposition}
 \begin{proof}
 Let $x\in \Omega(\cC).$ We denote by $\cC^{\geq 1}$ the subset of $\cC$ comprising elements with $|$-norm $\geq 1.$ For every $\varepsilon > 0$ let $D_\varepsilon$ be a small disc of size $\varepsilon$
 centered at $x$ and transverse to the orbit foliation $\cO,$ so that
 the restrictions $\cG^s$ and $\cG^u$ of the weak foliations to $D$ are product; any leaf of
 $\cG^s$ intersects every leaf of $\cG^u$ in one and only one point - the existence of such a disc is a corollary of Theorem \ref{thm:local product}.
We can assume that these discs are nested, i.e. that $D_\varepsilon \subset D_{\varepsilon'}$ if $\varepsilon \leq \varepsilon'.$
Let $U$ be a relatively compact symetric neighborhood
 of the identity in $G$ such that the application $U \times D_\varepsilon \to M$ mapping $(g, y)$ to $\phi^g(y)$ is a
 diffeomorphism onto its image.

Since $\mathcal C$ is an open  convex cone, reducing $U$ if necessary, one can assume that  for any $\mathfrak a \in \cC^{\geq 1}$ and any $g$ in the closure of $U,$ the product $g\exp(\mathfrak a)$ is Anosov.
Moreover, there is a uniform bound $C>0$ such that the restriction of $d\phi^{g\exp(\mathfrak a)}$ to $E^{ss}$ (respectively to $E^{uu}$) is $Ce^{-\lambda | \mathfrak a |}$-contracting (respectively $C^{-1}\exp(\lambda | \mathfrak a |)$-expanding).

 By hypothesis, for any $n$, there is an element $x_n$ in $D_{1/n}$ and an element $\mathfrak a_n$ in $\cC$ such that $\phi^{\mathfrak a_n}(x_n)$ lies in the neighborhood described above, i.e.
has the form $\phi^{g_n}(y_n)$ where $y_n$ lies in $D_{1/n}$ and $g_n$ in $U$.

We consider first the case where the real numbers $| \mathfrak a_n|$
are bounded from above, i.e. that, up to a subsequence, $\mathfrak a_n$ converge to some element $\mathfrak a$ of $\overline{\cC}^{\geq 1}.$ Since $\cC$ is strict, $\mathfrak a $ is Anosov. Up to a subsequence, $x$ is a fixed point
of  $g_\infty^{-1}\exp(\mathfrak a)$ where $g_\infty$ is a limit of the $g_n$'s. Hence, $x$ fixed by an Anosov element, and thus contained in Comp$(\phi)$. The Lemma is proved in this case.

We are thus reduced to the other case, i.e. the case where, up to a subsequence, the real numbers $| \mathfrak a_n|$ converge to $+\infty$. Consider the local unstable leaf $\cG^{s}(x_n)$: its image
by the first return map from a neighborhood of $x_n$ in $D_{1/n}$ into a neighborhood of $y_n$ in $D_{1/n}$ is contracted, and therefore this first return map is well-defined. The composition of this map with the holonomy along $\cG^u$ defines a contracting map from  $\cG^{s}(x_n)$ into itself, which therefore admits a fixed point. It means that there is an element $h_n$ of $U$ such that $h_n\exp(\mathfrak a_n)$
maps a point $z_n$ of $D_{1/n}$ to a point $p_n$ in its unstable leaf $\cG^{u}(z_n).$

Now apply the same argument now to the reversed subcone $-\cC$, and replacing $x_n$ by $p_n$: since the stable/unstable foliations are switched when we replace $\cC$ by its opposite,
we obtain that the first return map preserves a local stable leaf.

The intersection of theses local stable and unstable leaves  is now a point fixed by an element of the form $g\exp(\mathfrak a_n)$ with $g$ in $\bar{U}$, hence Anosov.
Therefore $D_{1/n}$ contains a point fixed by an Anosov element, hence lying in Comp$(\phi)$. The Proposition follows since $n$ is arbitrary.
 \end{proof}

 \begin{remark}
\label{rk:omegaegal}
 It follows by proposition \ref{prop:closCompactO} that we have:
$$\Omega(\mathfrak a) = \Omega(\cC) = \overline{\operatorname{Comp}(\phi)}$$
In particular, the nonwandering sets $\Omega(\mathfrak a)$  and $\Omega(\cC)$ are independent
 from the choice of the Anosov element $\mathfrak a,$ the Anosov subcone $\cC,$ and even of the preferred splitting $\xi$ and preferred Anosov chamber $\mathcal A_0.$
 Hence we can denote the nonwandering set simply by $\Omega(\phi)$ or $\Omega.$
 \end{remark}


We will need the following refinement of Theorem \ref{thm:compact_orbit}. Recall Definition \ref{def:subcone}:

 \begin{lemma}
 \label{le:compactcausal}
For any Anosov subcone $\cC$, and any compact orbit $\cO_x$, there is an element of $\cC$ fixing an element in $\cO_x$.
 \end{lemma}
\begin{proof}
It follows immediatly by repeating the argument used in the proof of Remark \ref{rmk:lattice_anosov}, replacing $\cA_0$ by $\cC$.
\end{proof}

\begin{proposition}
\label{pro:lambdastable}
Let $x,$ $y$ be two elements of $\Omega$. Assume that the intersections $\cF^u(x)\cap \cF^s(y)$ and $\cF^s(x)\cap \cF^u(y)$ are non-empty.
Then, these intersections are both contained in $\Omega$.
\end{proposition}

\begin{proof}
The argument is classical, but the proof more oftenly involves the notion of pseudo-orbits, that we prefer to avoid here (see for exemple \cite[Proposition $8.11$]{shubook}).
Since ${\rm Comp}(\phi)$  is dense in $\Omega,$ and by continuity of the maps $(x, y) \mapsto \cF^u(x)\cap \cF^s(y)$ (cf. Theorem \ref{thm:local product})
on can assume that $x$ and $y$ are in ${\rm Comp}(\phi)$.

Let $\cC$ be a strict Anosov subcone in our preferred Anosov chamber $\cA_0$ (hence elements of $\cC$ preserve the foliations $\cF^{ss}$ and $\cF^{uu}$).
By Lemma \ref{pro:lambdastable}, one can assume that $x$ is fixed by some element $\mathfrak a$ of $\cC,$ and $y$ by some element $\mathfrak b$ of $\cC.$

Let $p$ be a point in $\cF^{uu}(x)\cap \cF^s(y)$ and let $U$ be a small product neighborhood of $p.$ We aim to show that some non trivial element
of $\cC$ maps a point in $U$ into $U.$

For $k$ big enough, $D' = \phi_{-k\mathfrak a}(D)$ contains the point $p$. Reducing $U$ if necessary, one can assume that $U$
is the saturation under some arbitrarly small neighborhood $V$ of $e$ in $N$ of an open neighborhood $\overline{U}$ of $p$ in $D'$ (see Figure \ref{default}).

Let now $q$ be a point in $\cF^{ss}(x)\cap \cF^u(y)$. Similarly, some iterate $D'' = \phi_{\ell\mathfrak a}(D)$ with $k>0$ contains $q$. Let $W$ be a neighborhood of
$q$ as depicted in Figure \ref{default}. In this figure, we draw two hatched regions $W'$ and $W''$ which have the following properties:

-- $W' \subset \overline{U},$

-- $W'' \subset W,$

-- $W'' = \phi_{(k+\ell)\mathfrak a}(W'),$

-- $W'$ is saturated by the restriction $\cG_U^{s}$ of $\cF^s$ to $\overline{U},$

-- $W''$ is saturated by the restriction $\cG_W^{u}$ of $\cF^u$ to $W.$

We furthermore select $W$ so that it is a product neighborhood, i.e. such that every leaf of $\cG^s_W$ intersects every leaf of $\cG^u_W.$

\begin{figure}[htb]
\begin{center}
\includegraphics[scale=0.33]{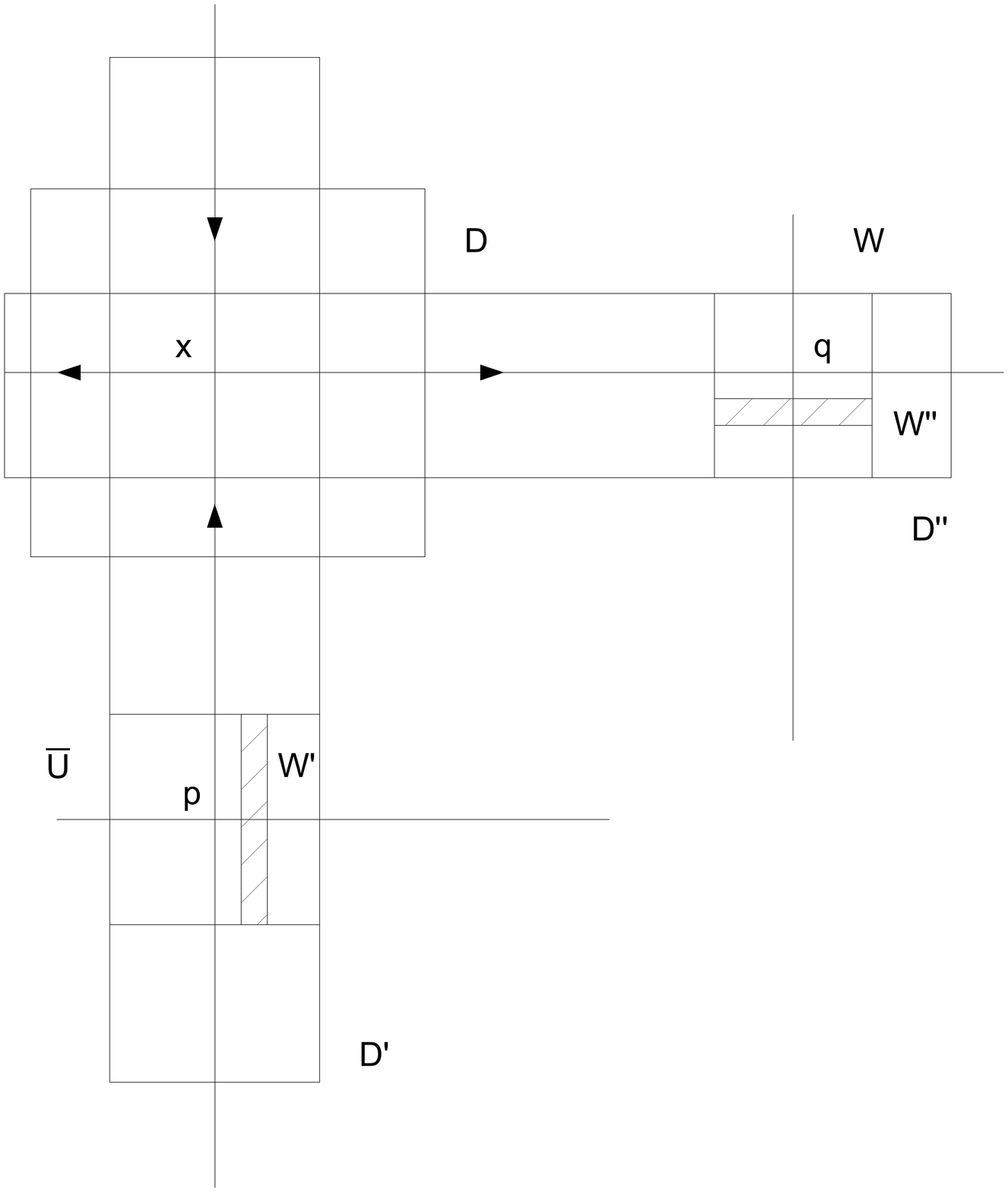}
\caption{Iterations of local sections.}
\label{default}
\end{center}
\end{figure}

We now replace $y$ by some $N$-iterate of itself so that $q$ lies in $\cF^{ss}(y).$  Then we have $p = g.p'$, where $p$ is a point in $\cF^{ss}(y) \cap \cF^u(x)$ and $g$ some element of $N.$
By repeating the argument above, we get a region $Z'$ in $W$, saturated by the foliation $\cG^s_W$, which is mapped by a positive iterate $\phi_{p\mathfrak b}$
to a region $Z''$ of $g^{-1}(\overline{U}).$

The closer $U$ is, the bigger are the integers $k$, $\ell$ and $p$. Hence one can assume that $g.\exp(p\mathfrak b)\exp((k+\ell)\mathfrak a)$ is the image under $\exp$ of some
element $\mathfrak c$ of $\cC^{\geq 1}.$
Since $W''$ is $\cG_W^u$-saturated, $Z'$ is $\cG_W^s$-saturated, and $W$ a product neighborhood, the intersection $W'' \cap Z'$ is non empty. It follows that any point
in $W' \cap \phi_{-(k+\ell)\mathfrak a}(W'' \cap Z')$ is mapped by $\phi_{\mathfrak c}$ to a pont in $gZ'' \subset \overline{U} \subset U$ (see Figure \ref{default2}).
The Proposition is proved.

\begin{figure}[htb]
\begin{center}
\includegraphics[scale=0.33]{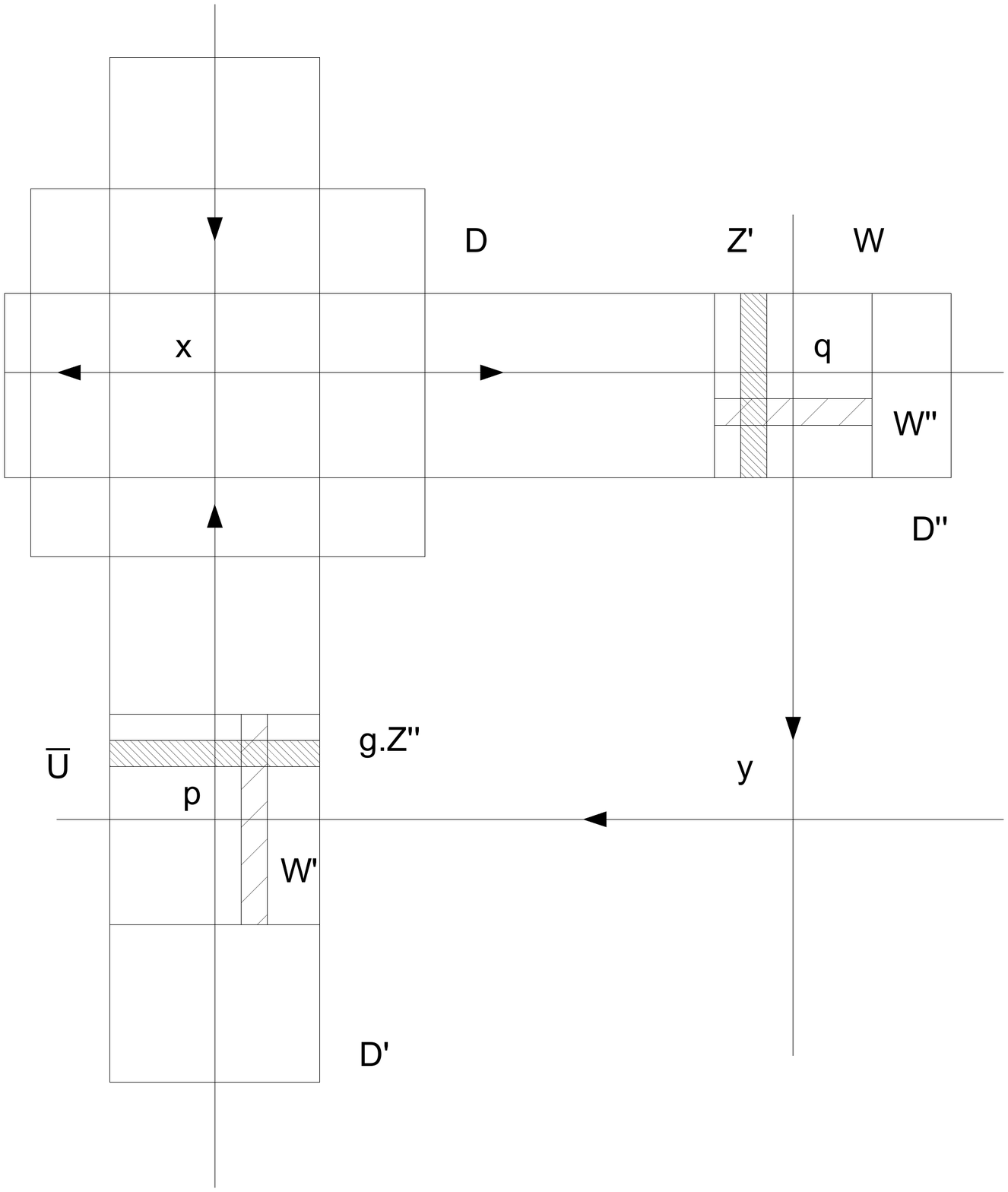}
\caption{A schematic view of the argument.}
\label{default2}
\end{center}
\end{figure}
\end{proof}

 \begin{theorem}[Spectral decomposition]
 \label{thm:espdecomp}
 Let $M$ be a closed smooth manifold and let $\phi$ be an Anosov action on $M$.
 The nonwandering set of $\phi$ can be partitioned into a finite number of
 $\phi$-invariant closed subsets, called basic blocks:
 $$
 \Omega = \bigcup_{i=1}^{\ell} \Lambda_i
 $$
 such that for every Anosov subcone $\cC$, every $\Lambda_i$ is $\cC$-transitive, i.e. for every open subsets $U$, $V$ in
 $M$ intersecting $\Lambda_i,$ there is an element $x$ of $U \cap \Lambda_i$ whose $\cC$-orbit in the meaning of Definition \ref{def:subcone} meets $V \cap \Lambda_i.$ In particular, there is an element of $\Lambda_i$ whose $\cC$-orbit is dense in $\Lambda_i.$
 \end{theorem}

 \begin{proof}
 Let ${\rm Comp}(\phi)$ be the set of compact orbits of $\phi$. By Proposition \ref{prop:closCompactO}
 we have $\overline{{\rm Comp}(\phi)}=\Omega(\phi).$
  We define a relation on ${\rm Comp}(\phi)$ by: $x\sim y$ if and
  only if $\cF^u(x)\cap \cF^s(y)\neq \emptyset$ and $\cF^s(x)\cap \cF^u(y)\neq \emptyset$
  with both intersections transverse in at least one point. Note that this relation is preserved by $\phi.$ We want
  to show that this is an equivalence relation and obtain each
  $\Lambda_i$ as the closure of an equivalence class.

  Note that $\sim $ is trivially reflexive and symmetric. In order to check
  the transitivity suppose that $x,y,z\in {\rm Comp}(\phi)$ and
  $p\in \cF^u(x)\cap \cF^s(y), $ $q\in \cF^u(y)\cap \cF^s(z)$ are
  transverse intersection points. There exists $\mathfrak a \in \mathcal{A}_0$ such that
  $ax=x$, where $a=\exp(\mathfrak a)$. Since the iterations under $\phi^a$ of a ball around $p$ in
  $\cF^u(p)=\cF^u(x)=\phi^a(\cF^u(x))$ accumulate on $\cF^u(y)$, we
 obtain that $\cF^u(x)$ and $\cF^s(z)$ have a transverse
 intersection. Analogously, we obtain that $\cF^s(x)$ and $\cF^u(z)$ have a transverse
  intersection.

 By Theorem \ref{thm:local product} any two sufficiently near points
 are equivalent, so by compactness we have finitely many equivalence
 classes whose (pairwise disjoint) closures we denote by
 $\Lambda_1,\Lambda_2,\dots, \Lambda_{\ell}$.

 It remains to show that every $\Lambda_i$ is $\cC$-transitive for every Anosov subcone $\cC \subset \cA_0.$
It means that for any two open sets
 $U$ and $V$ in $\Lambda_i$ there exists $\mathfrak a \in \cC$ such
 that $\phi^{\exp(\mathfrak a)}(U)\cap V\neq \emptyset$ - the existence of a dense $\cC$-orbit then follows by a classical argument
 (see for example \cite[Lemma 1]{brintop}).

 The density of compact orbits
 in $\Lambda_i$ implies the existence of $x \in U$ and $\mathfrak a \in \cC$
 such that $\phi_{\mathfrak a}(x)=x$ (cf. Lemma~\ref{le:compactcausal}).

Let $y$ be an element in $V \cap \Lambda_i \cap $ Comp$(\phi).$ We have $x \sim y,$ hence the stable leaf $\cF^{ss}(y)$ intersects
$\cF^u(x)$. Iterating by $\phi_{\mathfrak b}$ where $\mathfrak b$ is an Anosov element in $\cA_0$ fixing $y$, we observe as above that $\cF^{u}(x)$ contains points arbitrarly
close to $y$; in particular, a point $p$ in $V$.

Now, according to Remark 3 item $(2)$, there is an element $\mathfrak g$ such that $\exp(\mathfrak g)p$ belongs to $\cF^{uu}(x)$. Then, the negative iterates $p_n:=a^{-n}\exp(\mathfrak g)p$ converge to $x$.

Let $n$ be big enough so that:

-- $n\mathfrak a-\mathfrak g$ belongs to $\cC,$

-- $p_n:=a^{-n}\exp(\mathfrak g)p$ belongs to $U.$

By the Lie product formula, $\exp(-\mathfrak g/k + n{\mathfrak a}/k)^{k}$ for $k$ going to $+\infty$ converges to $\exp(-\mathfrak g)a^n$, hence
maps the point $p_n$ of $U$ to an element in $V.$ Since every $-\mathfrak g/k + n{\mathfrak a}/k$ is an element of $\cC$, there is as required
an element of $\cC$ mapping an element of $U$ to an element of $V.$

In order to conclude, we must show that this decomposition does not depend on our initial choice of preferred splitting $\xi$ and the chamber $\cA_0.$ Let $\xi'$ be another splitting
for which $\cA_{\xi'}$ is non-empty. We obtain then a spectral decomposition $ \Omega = \bigcup_{i=1}^{\ell} \Lambda'_i$ such that for every Anosov subcone $\cC' \subset\cA_{\xi'}$
every $\Lambda'_i$ is $\cC'$-transitive. Let $x$ and $y$ be two elements of the same basic set $\Lambda_i$, and let $\Lambda'_j$, $\Lambda'_k$ be the basic sets for  the splitting $\xi'$
containing respectively $x$ and $y$. Then, there are sequences $(x_n)_{n \in \mathbb N}$ and $(y_n)_{n \in \NN}$ in $\Omega$, the first converging to $x$, the second to $y$, and
such that $y_n$ lies in the $\cC$-orbit of $x_n.$ Then, since $\Lambda_i$ is isolated in $\Omega$, every $x_n$ and $y_n$ lies in $\Lambda_i$ for $n$ sufficiently big.
But for the same reason, $x_n$ lies in $\Lambda'_j$ and $y_n$ lies in $\Lambda'_k$ for $n$ sufficiently big. Since $\Lambda'_j$ is $G$-invariant, we have $\Lambda'_j = \Lambda'_k$.
Therefore, every $\xi$-basic set $\Lambda_i$ is contained in one $\xi'$-basic set $\Lambda'_j.$ The reverse inclusion is obtained by the same argument, exchanging the roles of
$\xi$ and $\xi'.$ The Theorem follows.
 \end{proof}

\begin{definition}
{\rm
For every basic set $\Lambda_i$ we denote by $\cF^s(\Lambda_i)$ the union of the leaves $\cF^s(x)$ where $x$ describes $\Lambda_i.$
}
\end{definition}

\begin{lemma}
\label{le:saturelambda}
For every basic set $\Lambda_i$ and any Anosov element $\mathfrak a$ of $\cA_0$, an element $x$ of $M$ lies in $\cF^s(\Lambda_i)$ if
and only if the distance between $\exp(t\mathfrak a)x$ and $\Lambda_i$ goes to zero when $t \to +\infty.$
\end{lemma}
\begin{proof}
If $x$ lies in $\cF^s(\Lambda_i)$, since $\Lambda_i$ is $\phi$-invariant,  $x$ lies in a strong stable leaf $\cF^{ss}(p)$ of some element in
$\Lambda_i.$ Then the distance $d(\exp(t\mathfrak a)x, \exp(t\mathfrak a)p)$ goes to $0$ when $t \to +\infty.$ It proves one implication of the statement.

Inversely, assume that $d(\Lambda_i, \exp(t\mathfrak a)x)$ goes to zero. Assume by a way of contradiction that $x$ does not belong to $\cF^s(\Lambda_i)$.
Then, for $t$ sufficiently big, $x_t:=\exp(t\mathfrak a)x$ lies in a neighborhood of $\Lambda_i$ where the local product
property holds: there is a point $p_t$ in $\Lambda_i$ such that $\cF^{uu}_\delta(x_t)$ intersects $\cF^{s}_\delta(p_t).$ Since $x \notin \cF^s(\Lambda_i)$, we have
$x_t \notin \cF^{s}(p_t).$ Let $r_t$ be the distance in $\cF^{uu}(x_t)$ between $x_t$ and $\cF^s(\Lambda_i)$. At one hand we have the uniform bound $r_t \leq \delta$ for any $t$ sufficiently big.
On the other hand, since $\cF^s(\Lambda_i)$ is $\exp(\mathfrak a)$-invariant, this distance increases exponentially with $t$. Contradiction.
\end{proof}

\begin{corollary}
\label{cor:Mlambda}
Every point in $M$ lies in the stable leaf (respectively unstable leaf) of a non-wandering point. In other words:
$$M = \bigcup_{i =1}^\ell \cF^s(\Lambda_i) =  \bigcup_{i =1}^\ell \cF^u(\Lambda_i)$$
\end{corollary}
\begin{proof}
Let $\mathfrak a$ be an element of $\cA_0.$ Let $V_1$, $V_2$, ... , $V_\ell$ be two by two disjoint open neighborhoods of the $\Lambda_i$'s.
The $\omega$-limit set of the $\phi_{\mathfrak a}$-orbit of $x$ is contained in $\Omega(\mathfrak a)$ and thus is covered by  the disjoint union
of the $V_i'$. Moreover, it is connected, hence it is contained in one and only one open domain $V_i$, hence in $V_i \cap \Omega(\mathfrak a) = \Lambda_i$.
Therefore, the distance between $\exp(t\mathfrak a)$ and $\Lambda_i$ converges to $0$. Then, it follows from Lemma \ref{le:saturelambda} that $x$
lies in $\cF^s(\Lambda_i)$.

The similar statement for the unstable foliation is proved by reversing the flow.
\end{proof}

\begin{remark}
The decomposition $ \Omega = \bigcup_{i=1}^{\ell} \Lambda_i$ is independant from the choice of the splitting $\xi,$ but of course it is not the case for the decomposition
$M = \bigcup_{i =1}^\ell \cF^s(\Lambda_i) =  \bigcup_{i =1}^\ell \cF^u(\Lambda_i).$
\end{remark}

We will need later the follwing criteria:

\begin{proposition}
 \label{pro:transitivity}
 If every leaf of $\cF^{s}$ is dense in $M$, then $\Omega(\phi) = M.$
 \end{proposition}

 \begin{proof}
By Theorem \ref{thm:espdecomp}, for any $p$ in $M$ there is $x, y \in \Omega(\phi)$ such that $p \in \cF^u(x) \cap \cF^{s}(y).$
The hypothesis implies that the stable leaf $\cF^s(x)$ intersects $\cF^u(y).$ Hence by Proposition \ref{pro:lambdastable} we have $p \in \Omega(\phi).$
 \end{proof}

\begin{definition}
\label{def:graph}
Let $\mathfrak G$ be the graph whose vertices are the basic sets $\Lambda_i,$ and such that there is an edge connecting $\Lambda_i$ and $\Lambda_j$ if and only if
the closure of  $\cF^u(\Lambda_i)$  contains $\Lambda_j$ (hence $\cF^{u}(\Lambda_j)$), or the closure of  $\cF^u(\Lambda_j)$ contains $\Lambda_i$ (hence $\cF^u(\Lambda_i)$).
\end{definition}

\begin{lemma}
\label{le:Gconnected}
$\mathfrak G$ is connected.
\end{lemma}

\begin{proof}
For every connected component $C$ of $\mathfrak C,$ let $F(C)$ be the union of $\cF^u(\Lambda_i)$ where $\Lambda_i$ describes the set of basic sets in $C.$
Let $x$ be an element in the closure of $F(C):$ since the number of basic sets is finite, $x$ is in the closure of $\cF^u(\Lambda_i)$
for some $\Lambda_i$ in $C.$ Let $\Lambda_j$ be the  basic set such that $x \in \cF^u(\Lambda_j)$ (cf. Corollary \ref{cor:Mlambda}).
Then $\cF^u(\Lambda_j)$ is contained in the closure of  $\cF^u(\Lambda_i)$, hence $\Lambda_i$ is in $C$.

It follows that $F(C)$ is closed.

Let now $C'$ be any other connected component of $\mathfrak G.$ Assume that $F(C) \cap F(C')$ is non-empty: let $x$ be an element in this intersection. Then $x$ lies in some $\cF^{u}(\Lambda_j),$ and one proves as above that $\Lambda_j$ lies in $C$ and in $C'.$ Hence $C=C'.$

In other words, $M$ is the disjoint union of the closed subsets $F(C).$ The Lemma follows from this fact and the connectedness of $M.$
\end{proof}






\subsection{Injectivity of the holonomy}

Recall that in Remark \ref{rk.zero} item $(3)$ we have defined, for every point $x$ and any element $\mathfrak h$ of $\mathcal N$ satisfying $(\exp\mathfrak h)\cF^{ss}(x) = \cF^{ss}(x)$ a
germ at $x$ of homeomorphism $h_x^{\mathfrak h}$ of $\cF^{uu}(x)$,
representing the holonomy transverse to the foliation $\cF^s$ of a certain loop $\omega_x^{\mathfrak h}$ obtained by composing the orbit of
$x$ under $\phi_{\mathfrak h}$ during the time $[0,1]$ and
any path from $\exp(\mathfrak h)x$ to $x$ in $\cF^{ss}(x).$ Moreover, every loop in the leaf $\cF^s(x)$ is homotopic to a loop
$\omega_x^{\mathfrak h}$, hence every element of the holonomy group of the leaf $\cF^s(x)$ is represented by some $h_x^{\mathfrak h}.$
We call $h_x^{\mathfrak h}$ the \textit{holonomy of ${\mathfrak h}$ at $x.$} Observe that if $\mathfrak h$ belongs to
$\cA_0$, then $h_x^{\mathfrak h}$ is expanding, in particular, non-trivial.

\begin{definition}
\label{def:irreducible}
{\rm
The Anosov action $\phi$ has \textit{injective stable holonomy} if for every leaf $F$ of $\cF^s$ the holonomy representation $\pi_1(F,x) \to$
Homeo$_{loc}(\cF^{uu}(x))$ is injective. According to the discussion above, it is equivalent to the following requirement:
for every non-trivial element $\mathfrak h$ of $\mathcal G$ and every $x \in M$ such that
$\exp(\mathfrak h)\cF^{ss}(x) = \cF^{ss}(x)$ the holonomy $h_x^{\mathfrak h}$ at $x$ is non-trivial.

Similarly, one define injective unstable holonomy; the action \textit{has injective holonomy} if it has stable and unstable injective holonomy.
}
\end{definition}

Of course, Nil-extensions do not have injective holonomy: loops contained in the fibers are loops contained in orbits with trivial holonomy.
The following proposition shows that it is the only obstruction to injective holonomy:

\begin{proposition}
 \label{le:trivial}
Faithful Anosov actions of nilpotent Lie groups have injective holonomy.
 \end{proposition}

All this section is devoted to the proof of Proposition \ref{le:trivial}. More precisely, we prove that the action has injective stable holonomy, the injectivity of the unstable holonomy is then obtained by simply replacing the Anosov element by its opposite. Actually, we will  prove the reverse statement, assuming that the action has no injective stable holonomy, and proving that then it is not faithful.

\begin{definition}\hspace{-0.3cm}.
\begin{enumerate}
\item {\rm
 Let $\mathfrak h$ be a non trivial element of $\mathcal N.$ The set of points $x$ satisfying $(\exp\mathfrak h)\cF^{ss}(x) = \cF^{ss}(x)$ and $h_x^{\mathfrak h}=$ id is
denoted by $\mathcal H(\mathfrak h).$
}
\item {\rm
For every $x$ in $M$, the elements $\mathfrak h$ of $\mathcal N$ satisfying $(\exp\mathfrak h)\cF^{ss}(x) = \cF^{ss}(x)$ and $h_x^{\mathfrak h}=$ id is denoted by $\Delta^0_x.$
}
\item {\rm
If $\Delta^0_x$ is non trivial, $x$ is said \textit{to admit trivial holonomy.}
}
\item {\rm the set of points of $M$ admitting trivial holonomy is denoted by $\mathcal P.$}
\end{enumerate}
\end{definition}

\begin{lemma}
\label{le:fssinvariant}
The set $\mathcal H(\mathfrak h)$ is $\cF^{ss}$-invariant.
\end{lemma}
\begin{proof}
Let $x$ be an element of $\mathcal H(\mathfrak h)$ and $y$ an element in $\cF^{ss}(x).$ If $y$ is sufficiently close to $x,$ the loop
$\omega_y^{\mathfrak h}$ considered in Remark \ref{rk.zero} in order to define $h^{\mathfrak h}_y$ can be chosen close to the loop $\omega_x^{\mathfrak h}.$
These loops are then
freely homotopic one to the other in $\cF^s(x)$. Hence we have $h^{\mathfrak h}_y=h^{\mathfrak h}_x$=id  and $y \in \mathcal H(\mathfrak h).$

The case for general $y$ in $\cF^{ss}(x)$ is obtained by iterating under $\exp(t{\mathfrak a_0}).$
\end{proof}

\begin{lemma}
\label{le:fuuinvariant}
The subset $\mathcal P$ is either empty, or the entire $M.$
\end{lemma}

\begin{proof}
Let $x$ be an element of $\mathcal P$, i.e. an element of $\mathcal H(\mathfrak h)$ for some non-trivial element $\mathfrak h.$
For every $g$ in $N,$ we have $gx \in \mathcal H({\rm ad}(g)\mathfrak h)$ hence $\mathcal P$ is $N$-invariant. It follows, together with Lemma \ref{le:fssinvariant}, that $\mathcal P$ is $\cF^s$-invariant.

Let now $y$ be an element of $\cF^{uu}(x).$ Since the holonomy $h^{\mathfrak h}_x$ is trivial, if $y$ is sufficiently close to $x$
in $\cF^{uu}(x),$  the image $(\exp\mathfrak h)y$ lies
in the same stable leaf $\cF^s(y)$ than $y.$ Moreover $\phi_{\mathfrak h}$ maps $\cF^{ss}(y)$ close to itself (since it preserves
$\cF^{ss}(x)$) and therefore we have $\cF^{ss}(y) = (\exp\mathfrak h')\cF^{ss}(y)$ for some element $\mathfrak h'$ close to $\mathfrak h.$ The loop $\omega_y^{\mathfrak h'}$
can be chosen so that it remains close to $\omega_x^{\mathfrak h}$; hence to be the lifting in $\cF^s(y)$ of $\omega_x^{\mathfrak h}$. Since $h^{\mathfrak h}_x$ is trivial,
it follows that $h^{\mathfrak h'}_y$ is also trivial, hence $y \in \mathcal H(\mathfrak h) \subset \mathcal P.$

Now for a general $y$ in $\cF^{uu}(x),$ not necessarily close to $x,$ some iterate $\exp(t\mathfrak a_0)y$ will
be close to  $\exp(t\mathfrak a_0)x$ in $\cF^{uu}(\exp(t\mathfrak a_0)x),$ hence in $\mathcal P$ anyway.

The set $\mathcal P$ is therefore invariant by the foliations $\cF^s$ and $\cF^u.$ The lemma follows.
\end{proof}

From now on, we assume $\mathcal P = M.$

\begin{lemma}
Let $x_0$ be an element of a compact $\phi$-orbit $\cO_0$. Then  $\exp(\Delta_x^0)$ is a normal subgroup of the $N$-stabilizer $\Delta_x.$
\end{lemma}

\begin{proof}
The key point is that the intersection between the strong stable leaf $\cF^{ss}(x_0)$ and $\cO_0$ is reduced to $x_0.$ Indeed, by the local product structure the intersection
between $\cO_0$ and $\cF^{ss}_\delta(x)$ is reduced to $x_0,$ and the general claim follows by iterating under (the inverse of) a $\xi$-Anosov element fixing $x_0.$

Therefore we have $\exp(\Delta_x^0) \subset \Delta_x.$ It is straightforward to check that $\exp(\Delta_x^0)$ is a subgroup.

The fact that $\exp(\Delta_x^0)$ is normal in the stabilizer follows easily from the fact that the holonomy $h^{{\rm ad}g(\mathfrak h)}_x$ is the conjugate under $\phi_g$ of
$h^{\mathfrak h}_x$, hence trivial if $h^{\mathfrak h}_x$ is trivial.
\end{proof}

Let $0 = N^{r+1} \lhd N^r \lhd ... \lhd N^1=[N, N] \lhd N^0=N$ the lower central serie for $N.$
For every $x$ in Comp$(\phi)$, let $k(x)$ be the biggest integer $k$ for which $\exp(\Delta_x^0) \cap N^k$ is not trivial. Observe that $k$ is $N$-invariant. We denote
by $\Delta_x^1$ the intersection $\exp(\Delta_x^0) \cap N^{k(x)}.$


\begin{lemma}
\label{le:xkfixed}
For every $x$ in Comp$(\phi)$ the weak unstable leaf $\cF^u(x)$ is fixed pointwise by $\Delta_x^1.$
\end{lemma}

\begin{proof}
Then, for every $g$ in $\Delta_x$ and every $\gamma$ in $\Delta_x^1,$
the commutator $g\gamma g^{-1}\gamma^{-1}$ lies in $\exp(\Delta_x^0) \cap N^{k+1},$ hence is trivial. It follows that $\Delta^1_x$ lies in the centralizer of $\Delta_x.$ Since $N$ is nilpotent and $\Delta_x$ a lattice in it, $\Delta_x^1$ is contained in the center $Z$ of $N$ (indeed, $\Delta_x^1$ lies in the center of the Zariski closure of $\Delta_x,$ which is the entire $N$, see Theorem $2.10$ in \cite{raghu}).

It follows that for every $h = \exp(\mathfrak h)$ in $\Delta_x^1$, the entire orbit $\cO(x)$ is fixed pointwise by $h.$ Moreover, for every $y$ in $\cO(x),$ the unstable leaf $\cF^{uu}(y)$ is preserved, and
fixed pointwise since the holonomy $h_y^{\mathfrak h}$ is trivial. The lemma follows.
\end{proof}


\begin{lemma}
Let $x_1$ and $x_2$ be two points in Comp$(\phi).$ If $x_1$ and $x_2$ belongs to the same basic set $\Lambda_i$, then we have $k(x_1) = k(x_2)$ and $\Delta_{x_1}^1 = \Delta^1_{x_2}.$
\end{lemma}

\begin{proof}
Let $y$ be a point in $\cF^u(x_1) \cap \cF^{ss}(x_2).$ Let $\mathfrak h$ be an element of $\Delta_{x_1}^0$ whose projection by $\exp$ is in $\Delta^1_{x_1}.$ According to Lemma \ref{le:xkfixed},
$y$ is a fixed point of $\exp(\mathfrak h)$, hence belongs to $\mathcal H(\mathfrak h).$ According to Lemma \ref{le:fssinvariant}, $x_2$ lies in $\mathcal H(\mathfrak h).$ Hence, $\mathfrak h$ lies in
$\Delta_{x_2}^0.$ Moreover, since $\exp(\mathfrak h) \in N^{k(x_1)}$, we have:
$$k(x_2) \geq k(x_1).$$

In a symetric way, one shows that every element in $\Delta_{x_2}^0$ with exponential in $\Delta^1_{x_2}$ is contained in $\Delta_{x_1}^0$ and that $k(x_1) \geq k(x_2).$ The lemma follows.
\end{proof}

It follows that for every basic set $\Lambda_i$, there is a non trivial discrete subgroup $\Delta_i$ of the center $Z$ of $N$ fixing $\Lambda_i$ pointwise. Actually, by Lemma \ref{le:xkfixed}, the entire $\cF^{u}(\Lambda_i)$ is contained in Fix$(\Delta_i).$

More precisely, there is an integer $k_i$ (the common value of $k$ on Comp$(\phi) \cap \Lambda_i$) such that:

-- no non trivial element of $N^{k_i+1}$ fixes pointwise $\cF^{u}(\Lambda_i)$ (indeed, such an element would be an element
of $\Delta_x^0 \cap N^{k(x)+1}$ for any $x$ in  Comp$(\phi) \cap \Lambda_i$),

-- $\Delta_i$ is exactly the subgroup of $N^{k_i}$ comprising the elements fixing  pointwise $\cF^{u}(\Lambda_i),$

Moreover, $\Delta_i$ is contained in the center $Z$ of $N.$

It follows from these properties that if $\Lambda_j$ is any basic set contained in the closure of $\cF^u(\Lambda_i)$ then
we have $k_j = k_i$ and $N^{k_i} = N^{k_j}.$ Since the graph $\mathfrak G$ is connected (Definition \ref{def:graph}, Lemma \ref{le:Gconnected}),
all the basic sets have the same integer $k_i$ and the same $N^{k_i}$.
It follows that elements of $N^{k_i}$ acts trivially on the entire $M.$ It is impossible if $\phi$ is faithful.

The proof of Proposition \ref{le:trivial} is complete.

\section{Actions of codimension one}
In this section, we keep considering an Anosov action $\phi: N \times M \to M$ of a simply connected nilpotent Lie group $N$, but we know add the additional hypothesis:
\medskip

\begin{center}
\textbf{$E^{uu}$ has dimension one.}
\end{center}



\subsection{Affine structures along the strong unstable leaves}
Most of the content of this section is classical. We refer to \cite[section $4.1$]{paper1} for more details.
An affine structure of class $C^2$ on $\RR$ is equivalent to the data of differential $1$-form on $\RR,$ hence once fixed a differential $1$-form $\omega_0$ on $M$ which does not vanish on $E^{uu},$ any continuous function $f: M \to \RR$ induces an affine structure along leaves of $\cF^{uu}:$ the one induced by the restriction of $f\omega_0$ on the leaf. We say that such an affine structure is \textit{$\phi$-invariant} if for every $a \in N$ and every $x \in M$ the restriction of $\phi^a$ induces an affine diffeomorphism bewtween $\cF^{uu}(x)$ and $\cF^{uu}(\phi^a(x)).$

\begin{proposition}
\label{pro:affine}
There exists an unique affine structure along $\cF^{uu}$ depending continuously on the points and invariant by the action $\phi.$
\end{proposition}

\begin{proof}
Let $\mathfrak a_0$ be an Anosov element of $\mathcal N$ expanding the leaves of $\cF^{uu}$. Then, a contracting fixed-point Theorem on a suitable Banach space ensures
the existence of a unique affine structure along $\cF^{uu}$
invariant by the one-parameter subgroup  $H_0$ generated by $\mathfrak a_0$ (see \cite{barthese} or  \cite[Proposition $3.4.1,$ Remarque $3.4.2$]{ghys3}).
Since it is unique, this affine structure is also $N(H_0)$-invariant. By an induction involving Lemma \ref{le:ascendingchain} we finally get that this affine structure is $N$-invariant.

For a proof of the completeness of the leaves, see \cite[\S $1.3$]{ghys3}.
\end{proof}

\begin{corollary}
\label{cor:linearizable}
Let $\mathfrak a$ be an element of $\mathcal N,$ and $x$ an element of $M$ such that $\phi_{\mathfrak a}(x) = x.$
Then the restriction of $\phi_{\mathfrak a}$ to $\cF^{uu}(x)$ is differentially linearizable. In particular, either $\cF^{uu}(x)$
pointwise fixed by $a$, or $x$ is an isolated fixed point of  $\phi_{\mathfrak a}$  in $\cF^{uu}(x).$\fin
\end{corollary}

\subsection{Anosov actions of codimension one are abelian}
An important corollary of the previous section is:

\begin{theorem}\label{thm:abelian}
Every codimension one Anosov action of a nilpotent Lie group is a Nil-extension over
a codimension one Anosov action of $\RR^k.$
\end{theorem}

\begin{proof}
According to Lemma \ref{le:centralextension}, such an Anosov action is a Nil-extension over a faithful Anosov action.
Consider a compact orbit $\Lambda \backslash N:$ the holonomy of the stable leaf
containing this orbit is linearizable (Corollary \ref{cor:linearizable}): it provides a morphism from $\Lambda$ into $\RR^\ast.$
According to proposition \ref{le:trivial}, this morphism is injective. The theorem follows.
\end{proof}

It follows that many results in \cite{paper1}, insensitive to Nil-extensions, still hold for Anosov actions
of nilpotent Lie groups of codimension one, not necessarily abelian. More precisely:

Let $\pi:\widetilde{M}\to M$ be the universal covering map of $M$ and $\widetilde{\phi}$ be the
 lift of $\phi$ on $\widetilde{M}$. The foliations $\cF^{ss},\ \cF^{uu},\ \cF^{s}$ and
 $\cF^{u}$ lift to foliations $\widetilde{\cF}^{ss},\ \widetilde{\cF}^{uu},\ \widetilde{\cF}^{s}$ and
 $\widetilde{\cF}^{u}$ in $\widetilde{M}$. We denote by $Q^{\phi}$ be the
 orbit space of $\tilde{\phi}$ and $\pi^{\phi}:\widetilde{M}\to Q^{\phi}$ be the
 canonical projection. We have the following properties:

 \begin{itemize}
   \item The orbits of ${\phi}$ are incompressible: every loop in a $\phi$-orbit $\cO$ which
 is homotopically non-trivial in $\cO$ is homotopically non-trivial in $M.$
   \item The  foliations $\widetilde{\cF}^{uu}$, $\widetilde{\cF}^{ss}$, $\widetilde{\cF}^{u}$,
 $\widetilde{\cF}^{s}$ and the foliation defined by $\widetilde{\phi}$ are by closed planes.
  The intersection between a leaf of $\widetilde{\cF}^{u}$ and a leaf of
 $\widetilde{\cF}^{s}$ is at most an orbit of $\widetilde{\phi}$.
 Every orbit of $\widetilde{\phi}$ meets a leaf of $\widetilde{\cF}^{uu}$ or $\widetilde{\cF}^{ss}$
 at most once.
   \item The universal covering of $M$ is diffeomorphic to $\RR^{n+k},$ and the orbit space $Q^\phi$ is is homeomorphic to $\RR^{n}.$
 \end{itemize}

Last but not least:

\begin{theorem}
\label{th:transitif}
Let $\phi$ be a codimension one Anosov action of a nilpotent Lie group $N$ of dimension $k$ on a closed manifold $M$
 of dimension $n+k$ with $n \geq 3.$ Then, $\Omega(\phi) = M$.
\end{theorem}

\begin{proof}
This is a consequence of Theorem $1$ in \cite{paper1} and of the observation that a Nil-extension of
a topologically transitive action is still topologically transitive. However, we propose here another proof, assuming
the theorem to be known for flows (cf. \cite{verjovsky}) and proving it in the case $n \geq 2.$

According to Theorem \ref{thm:abelian} one can assume that the group $N$ is $\RR^k$ with $n \geq 2.$
Then, since any subgroup of $\RR$ is either cyclic or dense, it follows from Corollary \ref{cor:linearizable} that
any stable leaf $F$ intersecting an unstable leaf ${\cF}^{uu}(\theta_0)$ of a periodic orbit
$\theta_0$ but with $\theta_0 \notin F$ is locally dense;
more precisely, $F$ is contained in the interior of the closure $\overline{F}$ of $F.$ On the other hand, it
follows from Corollary \ref{cor:Mlambda} that every leaf of $\cF^s$ satisfies this property.

Let $F$ be any leaf of $\cF^s$, and $G$ be a leaf in the closure $\overline{F}:$ this leaf
is contained in the interior of its closure $\overline{G},$ which is contained in $\overline{F}.$
Hence $G$ is contained in the interior of $\overline{F}.$ Therefore, $\overline{F}$ is closed and open, i.e. the entire $M.$
The Theorem then follows from Proposition \ref{pro:transitivity}.
\end{proof}

\subsection{Anosov actions of codimension two}
In the previous section we have seen that if $n \geq 3$, Anosov actions of codimension one are topologically transitive. In this section, elaborated with R. Var\~ao, we consider the remaining case $n=2.$ We show:

\begin{theorem}\label{thm:3}
Let $\phi$ be a codimension one Anosov action of a nilpotent Lie group $N$ of dimension $k$ on a closed manifold $M$
 of dimension $k+2.$ Then $\phi$ is a Nil-extension over an Anosov flow on a $3$-manifold.
\end{theorem}

This Theorem is a direct consequence of Theorem \ref{thm:abelian} and the main result of V. Arakawa in his (unpublished) Ph. D Thesis (\cite{arakawa}).
However, we propose here a proof much shorter than the one appearing in \cite{arakawa}.

\begin{proof}
By a way of contradiction, assume that we have a codimension one Anosov action of a nilpotent Lie group $N$ of dimension $k$ on a closed manifold $M$
of dimension $k+2$ that is not a Nil-extension over an Anosov flow on a $3$-manifold. By Theorem \ref{thm:abelian} one can assume that $N$
is $\RR^k$ with $k \geq 2.$ Let $\theta_0$ be an element of the orbit space $Q^\phi$ which is a lift of a compact orbit: $\theta_0$ is fixed
by a subgroup $\Lambda$ of $\pi_1(M)$ isomorphic to a lattice in $\RR^k,$ i.e. isomorphic to $\ZZ^k.$ According to Corollary \ref{cor:linearizable}
and the local product structure near $\theta_0,$ the action of $\Lambda$ in $Q^{\phi}$ near $\theta_0$ is linearizable. More precisely:
there is a neighborhood $U$ of $\theta_0$ in $Q^{\phi},$ morphisms $\rho_{1,2}: \Lambda \to ]0, +\infty[$ and a homeomorphism $\psi: U \to V \subset \cG^s(\theta_0) \times \cG^u(\theta_0) \approx \RR^2$ such that:
$$ \forall \gamma \in \Lambda, \forall \theta \in U \;\; \psi(\gamma.\theta) = (e^{\rho_1(\gamma)}\psi_1(\theta), e^{\rho_2(\gamma)}\psi_2(\theta))$$
where $\psi(\theta) = (\psi_1(\theta), \psi_2(\theta)).$

We denote by $\rho(\Gamma)$ the subgroup of $\RR^2$ comprising elements of the form $(\rho_1(\gamma), \rho_2(\gamma))$ with $\gamma \in \Lambda.$
We have two cases to consider:

\begin{enumerate}
  \item Either $\rho(\Gamma)$ is discrete in $\RR^2$ (hence a lattice since $k \geq 2$),
  \item Or $\rho(\Gamma)$ is not closed in $\RR^2.$
\end{enumerate}

In the first case, $\theta_0$ lies in the closure of the orbit under $\Lambda$ of any point in $U \setminus \{ \theta_0 \}.$ In
the second case, the orbit under $\Lambda$ of any point $\theta'$ in $U \setminus \{ \theta_0 \}$ accumulates non-trivially to $\theta'.$

In both situations, we conclude that no element of $U \setminus \{ \theta_0 \}$ can be a lift of a compact orbit of $\phi$. It follows
that the basic set containing the projection of $\theta_0$ is reduced to one compact orbit.

Since $\theta_0$ was arbitrary, we conclude that the nonwandering set $\Omega(\phi)$ is a finite union of compact orbits. According
to Corollary \ref{cor:Mlambda} the foliation $\cF^s$ has only finitely many leaves: contradiction.
\end{proof}


\end{document}